\documentclass[11pt]{amsart}

\usepackage{amsmath, amsthm, amssymb, amsfonts, amscd, verbatim, setspace}

\def\CC{{\mathbb C}}

\def\KK{{\mathbb K}}
\def\MM{{\mathbb M}}
\def\OO{{\mathbb O}}

\def\PP{{\mathbb P}}

\def\RR{{\mathbb R}}

\def\ZZ{{\mathbb Z}}

\def\Bsf{{\mathsf B}}
\def\Esf{{\mathsf E}}

\def\Psf{{\mathsf P}}

\def\hhat{{\hat h}}

\def\0{{\mathbf 0}}
\def\1{{\mathbf 1}}

\def\Ccal{{\mathcal C}}

\def\Ecal{{\mathcal E}}

\def\Ocal{{\mathcal O}}

\def\Qcal{{\mathcal Q}}
\def\Rcal{{\mathcal R}}
\def\Scal{{\mathcal S}}

\def\Kbar{{\bar K}}

\def\Aut{\mathrm{Aut}}

\def\diag{\mathrm{diag}}

\def\diam{\mathrm{diam}}

\def\ns{\mathrm{ns}}

\def\tor{\mathrm{tor}}

\theoremstyle{plain}

\newtheorem{thm}{Theorem}
\newtheorem{cor}[thm]{Corollary}
\newtheorem{prop}[thm]{Proposition}
\newtheorem{lem}[thm]{Lemma}

\reversemarginpar           

\begin{document}

\title[Non-archimedean Equidistribution on Elliptic Curves]{Non-archimedean Equidistribution on Elliptic Curves with Global Applications}
\author{Clayton Petsche}
\address{Department of Mathematics and Statistics; Hunter College; 695 Park Avenue; New York, NY 10065 U.S.A.}
\email{cpetsche@hunter.cuny.edu}
%\urladdr{...}
%\dedicatory{...}
\date{First version December 3, 2007; latest revision April 8, 2009.}
%\thanks{...}
%\keywords{...}
\begin{abstract}
Let $\KK$ be an algebraically closed, complete, non-archimedean field, let $E/\KK$ be an elliptic curve, and let $\Esf$ denote the Berkovich analytic space associated to $E/\KK$.  We study the $\mu$-equidistribution of finite subsets of $E(\KK)$, where $\mu$ is a certain canonical unit Borel measure on $\Esf$.  Our main result is an inequality bounding the error term when testing against a certain class of continuous functions on $\Esf$.  We then give two applications to elliptic curves over global function fields: we prove a function field analogue of the Szpiro-Ullmo-Zhang equidistribution theorem for small points, and a function field analogue of a result of Baker-Ih-Rumely on the finiteness of $S$-integral torsion points.  Both applications are given in explicit quantitative form.
\end{abstract}

\maketitle

%%%%%%%%%%%%%%%%%%%%%%%%%%%%%%%%%%%%%%%%%%%%%%%%%%%%%%%%%%%%
%%%%%%%%%%%%%%%%%%%%%%%%%%%%%%%%%%%%%%%%%%%%%%%%%%%%%%%%%%%%
%%%%%%%%%%%%%%%%%%%%%%%%%%%%%%%%%%%%%%%%%%%%%%%%%%%%%%%%%%%%

\section{Introduction}

\subsection{The main local inequality}

Let $\KK$ be a field which is algebraically closed and complete with respect to a non-trivial, non-archimedean absolute value $|~\cdot~|$, and let $E/\KK$ be an elliptic curve.  The main result of this paper is an inequality which measures the equidistribution of finite subsets of $E(\KK)$ with respect to a certain canonical unit Borel measure $\mu$ defined on the Berkovich analytic space $\Esf$ associated to $E/\KK$.  

We will give a thorough review of the space $\Esf$ in $\S$\ref{BerkAnalyticSection}, but briefly it is a path-connected, compact, Hausdorff topological space which contains $E(\KK)$ (with the topology induced by the absolute value on $\KK$) as a dense subspace.  The set $\Esf\setminus E(\KK)$ is endowed with a canonical path metric $\rho(x,y)$ giving it the structure of an (infinite) metrized graph.  In particular, there exists a distinguished finite subgraph $\Sigma$ of $\Esf$ called the {\em skeleton}, along with a strong deformation retraction $r_\Sigma:\Esf\to\Sigma$, such that each connected component of $\Esf\setminus\Sigma$ is homeomorphic to a simply connected open Berkovich disc $\Bsf^\circ$.  If the $j$-invariant $j_E$ is integral ($|j_E|\leq1$) the set $\Sigma$ consists of a single point, and so in particular $\Esf$ is simply connected.  If $j_E$ is not integral ($|j_E|>1$) the set $\Sigma$ is isometric to a circle of circumference $\log|j_E|$, and in this case the fundamental group of $\Esf$ is infinite cyclic.  The canonical unit Borel measure $\mu$ on $\Esf$ is supported on $\Sigma$, and is defined as follows: if $j_E$ is integral $\mu$ is the Dirac measure supported at the point $\Sigma$, while if $j_E$ is not integral $\mu$ is the normalized uniform measure supported on the circle $\Sigma$.
 
Let $Z$ be a nonempty finite subset of $E(\KK)$.  In $\S$\ref{MainLocalSection} we will define a nonnegative real number $D(Z)$ called the local discrepancy of $Z$, which is closely related to the N\'eron local height function on $E(\KK)$, and which is small precisely when the set $Z$ is close to being $\mu$-equidistributed.  Let $\Ccal(\Esf,\RR)$ denote the space of continuous real valued functions on $\Esf$.  Our main local result is summarized in the following theorem.

\begin{thm}\label{MainLocalThmIntro}
There exists a dense subspace $\Scal(\Esf,\RR)$ of $\Ccal(\Esf,\RR)$ such that 
\begin{equation*}
\bigg|\frac{1}{|Z|}\sum_{P\in Z}F(P)-\int F d\mu\bigg| \leq C_1(F)\bigg(D(Z) + \frac{C_2(F)}{|Z|}\bigg)^{1/2}
\end{equation*}
for all $F\in\Scal(\Esf,\RR)$ and all nonempty finite subsets $Z$ of $E(\KK)$, where $C_1(F)$ and $C_2(F)$ are constants depending only on $F$.
\end{thm}

Roughly, $\Scal(\Esf,\RR)$ is the space of continuous functions $F:\Esf\to\RR$ whose derivative $F'$ is supported on some finite subgraph of $\Esf$ and is square-integrable.  The constant $C_1(F)$ is essentially the $L^2$-norm of $F'$, and $C_2(F)$ is a certain quantity measuring the size of the support of $F'$.  In $\S$\ref{MainLocalSection} we will restate this result with precise definitions of $\Scal(\Esf,\RR)$, $C_1(F)$, and $C_2(F)$.  As applications of Theorem \ref{MainLocalThmIntro} we prove several global results which are summarized in the following two sections.

\subsection{The global equidistribution theorem}

Let $K$ be a global function field; we will state precisely what this means in $\S$\ref{GlobalSection}, but an example is the field $K=k(C)$ of rational functions on an integral, proper, geometrically connected curve $C$ defined over an arbitrary field $k$.  At each place $v$ of $K$, let $K_v$ be the completion of $K$ at $v$ and let $\KK_v$ be the completion of the algebraic closure of $K_v$; the field $\KK_v$ is both complete and algebraically closed (\cite{BGR}, $\S$3.4).  

Let $E/K$ be an elliptic curve, and at each place $v$ of $K$ let $\Esf_v$ be the Berkovich analytic space associated to $E/\KK_v$; more generally, we affix a subscript $v$ to all local objects associated to $E/\KK_v$, including the canonical measure $\mu_v$ on $\Esf_v$.  Recall that the N\'eron-Tate canonical height function $\hhat:E(\Kbar)\to\RR$ is nonnegative and vanishes on the torsion subgroup $E(\Kbar)_\tor$ of $E(\Kbar)$.  Given a finite set $Z$ of points in $E(\Kbar)$, define the height of $Z$ to be the average $\hhat(Z)=\frac{1}{|Z|}\sum_{P\in Z}\hhat(P)$ of the heights of its points.  Using Theorem \ref{MainLocalThmIntro} we will prove the following inequality measuring the local $\mu_v$-equidistribution of a large set $Z$ of small global points in $E(\Kbar)$.

\begin{thm}\label{QuanSUZIntro}
Let $E/K$ be an elliptic curve over a global function field, let $v$ be a place of $K$, and let $\epsilon:\Kbar\hookrightarrow\KK_v$ be a $K$-embedding.  There exists a dense subspace $\Scal(\Esf_v,\RR)$ of $\Ccal(\Esf_v,\RR)$ such that 
\begin{equation*}
\bigg|\frac{1}{|Z|}\sum_{P\in Z}F(\epsilon(P))-\int F d\mu_v\bigg| \leq C_1(F)\bigg(4\hhat(Z) + \frac{h(j_E)}{12|Z|} + \frac{C_2(F)}{|Z|}\bigg)^{1/2}
\end{equation*}
for all $F\in \Scal(\Esf_v,\RR)$ and all nonempty finite $\Aut(\Kbar/K)$-stable subsets $Z$ of $E(\Kbar)$, where $h(j_E)$ denotes the absolute Weil height of the $j$-invariant of $E$, and where $C_1(F)$ and $C_2(F)$ are the same constants associated to the function $F\in \Scal(\Esf_v,\RR)$ as in Theorem~\ref{MainLocalThmIntro}.
\end{thm}

Using the fact that $\Scal(\Esf_v,\RR)$ is dense in $\Ccal(\Esf_v,\RR)$, we will deduce the following corollary, which is a function field analogue of the Szpiro-Ullmo-Zhang equidistribution theorem for elliptic curves.

\begin{cor}\label{MainGlobalCorIntro}
Let $E/K$ be an elliptic curve over a global function field, let $v$ be a place of $K$, and let $\epsilon:\Kbar\hookrightarrow\KK_v$ be a $K$-embedding.  Let $\{Z_n\}_{n=1}^{\infty}$ be a sequence of finite $\Aut(\Kbar/K)$-stable subsets of $E(\Kbar)$ with $\hhat(Z_n)\to0$ and $|Z_n|\to+\infty$.  Then
\begin{equation}\label{MainGlobalCorLimForm}
\lim_{n\to+\infty}\frac{1}{|Z_n|}\sum_{P\in Z_n}F(\epsilon(P))=\int F d\mu_v
\end{equation}
for all continuous functions $F:\Esf_v\to\RR$.
\end{cor}

We will prove Theorem~\ref{QuanSUZIntro} and Corollary~\ref{MainGlobalCorIntro} in $\S$\ref{GlobalSection}.  

\subsection{The finiteness of $S$-integral torsion points}

Let $K$ be a global function field or a number field, and let $E/K$ be an elliptic curve given by a Weierstrass equation
\begin{equation}\label{Weierstrass}
y^2 + a_1xy + a_3y = x^3 + a_2x^2 + a_4x + a_6.
\end{equation}
Let $S$ be a finite set of places of $K$ (including the archimedean places if $K$ is a number field) such that $(\ref{Weierstrass})$ is defined over the ring $\Ocal_S = \{a\in K \mid |a|_v\leq1 \text{ for all places }v\not\in S\}$ of $S$-integers in $K$.  Given two points $P,Q\in E(\Kbar)$, we say that {\em $P$ is $S$-integral with respect to $Q$} if the Zariski closures of $P$ and $Q$ do not meet in the model $\Ecal/\Ocal_S$ of $E$ associated to the Weierstrass equation $(\ref{Weierstrass})$.  

For fixed $Q\in E(\Kbar)$, let $\Ecal(\bar{\Ocal}_S,Q)$ denote the set of points in $E(\Kbar)$ which are $S$-integral with respect to $Q$, and let $\Ecal(\bar{\Ocal}_S,Q)_\tor=E(\Kbar)_\tor\cap\Ecal(\bar{\Ocal}_S,Q)$ be the set of {\em torsion} points which are $S$-integral with respect to $Q$.  In the number field case, Baker-Ih-Rumely \cite{BakerIhRumely} have shown that if $Q\in E(\Kbar)$ is a nontorsion point, then $\Ecal(\bar{\Ocal}_S,Q)_\tor$ is finite.  More generally, Ih has conjectured analogous finiteness results for torsion points on abelian varieties and preperiodic points for dynamical systems on $\PP^1$ over number fields.  We will prove the following function field analogue of Baker-Ih-Rumely's theorem for elliptic curves.

\begin{thm}\label{BakerIhRumelyFF}
Let $E/K$ be an elliptic curve over a global function field.  Let $(\ref{Weierstrass})$ be a Weierstrass equation for $E/K$, let $S$ be a finite set of places of $K$ such that $(\ref{Weierstrass})$ is defined over $\Ocal_S$, and let $\Ecal/\Ocal_S$ denote the associated $S$-integral model for $E$.  If $Q$ is a point in $E(\Kbar)$ with $\hhat(Q)>0$, then the set $\Ecal(\bar{\Ocal}_S,Q)_\tor$ is finite.
\end{thm}

In the number field case, Baker-Ih-Rumely \cite{BakerIhRumely} treat the non-archimedean places using a result of Chambert-Loir (which in our notation is equivalent to the equidistribution of torsion points after they have been retracted onto the skeleton $\Sigma_v$), in combination with a result of Cassels bounding the denominators of torsion points in characteristic zero.  Our proof over function fields, which holds in arbitrary characteristic, requires the full equidistribution theorem on $\Esf_v$, and in place of Cassels' inequality it uses the discreteness of torsion points in $E(\KK_v)$.  On the other hand, the lack of archimedean places makes the proof of Theorem \ref{BakerIhRumelyFF} quite a bit less complicated than its number field analogue: the treatment of the archimedean places in \cite{BakerIhRumely} requires a strong quantitative equidistribution result on torsion points in $E(\CC)$, along with an effective diophantine inequality on linear forms in elliptic logarithms due to David--Hirata-Kohno.

Using Theorem~\ref{QuanSUZIntro} we can prove a quantitative version of Theorem~\ref{BakerIhRumelyFF}.  First observe that for fixed $Q\in E(\Kbar)$, the set $\Ecal(\bar{\Ocal}_S,Q)$ depends on the choice of Weierstrass equation and the set $S$.  However, enlarging the set $S$ only makes the set $\Ecal(\bar{\Ocal}_S,Q)_\tor$ larger, and since any two Weierstrass equations for $E/K$ give rise to isomorphic models over $\Ocal_S$ for some (sufficiently large) finite set $S$ of places of $K$, we see that the finiteness of $\Ecal(\bar{\Ocal}_S,Q)_\tor$ for all $S$ is independent of the Weierstrass equation $(\ref{Weierstrass})$.  Moreover, by enlarging the set $S$ if necessary, we may assume that the Weierstrass equation $(\ref{Weierstrass})$ has good reduction outside $S$.  Under this additional hypothesis we can give the following explicit bound on $\Ecal(\bar{\Ocal}_S,Q)_\tor$.

\begin{thm}\label{BakerIhRumelyFF2Intro}
Let $E/K$, $S$, and $\Ecal/\Ocal_S$ satisfy the same hypotheses as in Theorem \ref{BakerIhRumelyFF}, and assume in addition that $S$ contains all places $v$ of $K$ such that $|\Delta|_v\neq1$, where $\Delta\in K$ is the discriminant of $(\ref{Weierstrass})$.  If $Q$ is a point in $E(\Kbar)$ with $\hhat(Q)>0$, then 
\begin{equation}\label{BakerIhRumelyFF2IneqIntro}
|\Ecal(\bar{\Ocal}_S,Q)_\tor| \leq \frac{1}{\hhat(Q)^{2}}\bigg(\frac{|S|h(j_E)}{12}+\sum_{v\in S} m_v(Q)\bigg)^2.
\end{equation}
\end{thm}

In the inequality $(\ref{BakerIhRumelyFF2IneqIntro})$, $h(j_E)$ denotes the absolute Weil height of the $j$-invariant of $E$, and for each $v\in S$, $m_v(Q)$ is a certain nonnegative quantity which is large when some embedding of $Q$ into $E(\KK_v)$ is $v$-adically close to a torsion point; we will define this precisely in $\S$\ref{SIntSection}.  We point out that Theorem \ref{BakerIhRumelyFF2Intro} holds even when $S$ is empty; in that case $E$ must be defined over the constant field $K_0$ of $K$, and our proof shows that $\Ecal(\bar{\Ocal}_S,Q)_\tor$ is empty.  

\subsection{Other work}

It is now well-understood that Berkovich's theory of analytic spaces provides a suitable framework under which to prove non-archimedean equidistribution theorems of Szpiro-Ullmo-Zhang type; see for example \cite{BakerRumely}, \cite{ChambertLoir}, \cite{FavreRiveraLetelier}, and \cite{Thuillier}.  Until very recently such results have been worked out in the number field case only.  However, independently Faber \cite{Faber} and Gubler \cite{Gubler} have now proved a general equidistribution result for dynamical systems on varieties over function fields.  When $K=k(C)$ the results of Faber and Gubler contain our Corollary \ref{MainGlobalCorIntro} as a special case; on the other hand, we treat a wider class of function fields than in \cite{Faber}, and our results are quantitative.

Favre--Rivera-Letelier \cite{FavreRiveraLetelier} have proved a quantitative equidistribution result for small points with respect to dynamical systems on $\PP^1$ over number fields.  The local results in \cite{FavreRiveraLetelier} are proved using potential theory on the Berkovich projective line $\Psf^1$ and work for a large class of Borel measures on $\Psf^1$, while our techniques are more elementary and are formulated specifically for the {\em canonical} measure $\mu$ on $\Esf$.  However, Thuillier \cite{Thuillier} has developed potential theory on an arbitrary non-archimedean curve; it would be interesting to use this theory to obtain more general equidistribution results in the spirit of Theorem \ref{MainLocalThmIntro}.

Baker-Petsche \cite{BakerPetsche} established quantitative equidistribution results for elliptic curves over number fields, although they did not prove an explicit inequality on the error term as in Theorem \ref{QuanSUZIntro}.  Rather, the main inequality of \cite{BakerPetsche} is a bound on the local discrepancy $D_v(Z)$ at each place $v$, which is enough to deduce the qualitative equidistribution theorem and several quantitative corollaries.  As is often the case with non-archimedean equidistribution theorems of this type, the countability of the residue field of $\CC_v$ (over a number field) is used in \cite{BakerPetsche} to show that the limiting measure is supported on the skeleton $\Sigma_v$ of $\Esf_v$.  Our proof of Theorem~\ref{MainGlobalCorIntro} gets around this assumption by establishing the limit formula (via Theorem~\ref{QuanSUZIntro}) for the dense class $\Scal(\Esf_v,\RR)$ of test functions.  Note that the field $\KK_v$ over a function field $K$ may have uncountable residue field.

Our analytic treatment of $\Esf$ using the path metric $\rho(x,y)$ is heavily influenced by the work of M. Baker and R. Rumely, especially their monograph \cite{BakerRumely2} on the analytic theory of the Berkovich projective line $\Psf^1$.

\subsection{Plan of the paper}

In $\S$\ref{BerkAnalyticSection} we will give a detailed review of the topological and analytic structure of the Berkovich analytic space associated to an elliptic curve.  We prove our main local result (Theorem \ref{MainLocalThmIntro}) in $\S$\ref{MainLocalSection}, and in $\S$\ref{GlobalSection} and $\S$\ref{SIntSection} we treat the global applications.

%%%%%%%%%%%%%%%%%%%%%%%%%%%%%%%%%%%%%%%%%%%%%%%%%%%%%%%%%%%%
%%%%%%%%%%%%%%%%%%%%%%%%%%%%%%%%%%%%%%%%%%%%%%%%%%%%%%%%%%%%
%%%%%%%%%%%%%%%%%%%%%%%%%%%%%%%%%%%%%%%%%%%%%%%%%%%%%%%%%%%%

\section{Non-archimedean Preliminaries}\label{BerkAnalyticSection}

\subsection{Overview}

Let $\KK$ be a field which is algebraically closed and complete with respect to a non-trivial, non-archimedean absolute value $|\cdot|$.  Let $\OO =\{a\in \KK \mid |a|\leq1 \}$ and $\MM =\{a\in \KK \mid |a|<1 \}$ be the ring of integers in $\KK$ and its maximal ideal, respectively, and let $k=\OO/\MM$ denote the residue field, which is algebraically closed (\cite{BGR}, $\S$3.4).

Let $E/\KK$ be an elliptic curve with $j$-invariant $j_E\in\KK$.  In this section we will give a fairly detailed review of the Berkovich analytic space $\Esf$ associated to $E/\KK$, its canonical path metric $\rho(x,y)$, and its relation to the N\'eron local height function $\lambda$ on $E(\KK)$.  This is partly in order to make this paper as self-contained as possible; in particular, to understand our results it is not necessary to have any prior knowledge of the general theory of Berkovich analytic spaces.  Moreover, we intend to make essential use of the metrized graph structure on $\Esf$, and we will put special emphasis the completely explicit nature of our main inequality and its corollaries.  We will therefore need to have an explicit description of the space $\Esf$ and its canonical path metric $\rho(x,y)$.  See \cite{Berkovich} for the theoretical foundation of Berkovich analytic spaces in general, and see \cite{BakerRumely2} and \cite{Baker1} for more detailed expositions on the Berkovich projective line and more general curves.

\subsection{Elliptic curves and integral models}\label{BerkSection}

In this section we will fix a model $\Ecal/\OO$ associated to an integral Weierstrass equation for $E/\KK$, and we will record a basic lemma.  In later sections we will use this model to explicitly construct the Berkovich analytic space $\Esf$ and its associated path metric $\rho(x,y)$, working out the details in the cases of integral and non-integral $j$-invariant separately.  We will then explain the topology on $\Esf$ and its metrized graph structure in $\S$\ref{BerkTopology}.

First suppose that $j_E$ is integral ($|j_E|\leq1$).  Then there exists a Weierstrass equation 
\begin{equation}\label{WeierstrassEq}
y^2 + a_1xy + a_3y= x^3 + a_2x^2 + a_4x + a_6
\end{equation}
for $E$ with integral coefficients $a_1, a_3, a_2, a_4, a_6\in\OO$ and unit discriminant $\Delta\in\OO^\times$.  Letting $\Ecal/\OO$ denote the associated integral model for $E$, the special fiber $\bar{\Ecal}$ is an elliptic curve over $k$,  and the reduction map $\pi:E(\KK)\to\bar{\Ecal}(k)$ is a surjective group homomorphism.  In this situation $E/\KK$ is said to have good reduction.  

Suppose now that $j_E$ is not integral ($|j_E|>1$).  Then we let $\Ecal/\OO$ denote the integral model for $E$ associated to the Weierstrass equation 
\begin{equation}\label{TateEq}
y^2 + xy = x^3 + a_4(q)x + a_6(q)
\end{equation}
afforded by Tate's uniformization theory; see \cite{Tate} and \cite{SilvermanII} $\S$~V.3.  Here $q\in\KK^\times$ is a uniformizing parameter satisfying $|q|=|1/j_E|<1$, and $a_4(q), a_6(q)$ are certain elements of $\KK$ defined by convergent integral power series in $q$.  In particular we have a group isomorphism $\phi:\KK^\times/q^\ZZ\stackrel{\sim}{\to}E(\KK)$ given by $\phi(u)=(X(u,q), Y(u,q))$ for $u\not\in q^\ZZ$, where $X(u,q)$ and $Y(u,q)$ can be given explicitly in terms of $u$ and $q$.  Let $\pi:E(\KK)\to\bar{\Ecal}(k)$ denote the reduction map onto the special fiber (which is singular).  The set $\bar{\Ecal}_{\ns}(k)$ of nonsingular points on the special fiber is a group variety which is isomorphic to $k^\times$, and thus $E/\KK$ is said to have (split) multiplicative reduction.  Letting $E_0(\KK)=\pi^{-1}(\bar{\Ecal}_{\ns}(k))$ denote the set of points with nonsingular reduction, the map $\pi$ restricts to a surjective group homomorphism $\pi_0:E_0(\KK)\to\bar{\Ecal}_{\ns}(k)\simeq k^\times$.  Note that the set $\{u\in\KK^\times\mid |q|<|u|\leq1\}$ is a fundamental domain for the quotient $\KK^\times/q^\ZZ$.  Given an element $u$ in this domain, we have $\phi(u)\in E_0(\KK)$ if and only if $|u|=1$.  Moreover, when $|u|=1$, the point $\phi(u)\in E_0(\KK)$ reduces to the identity element of $\bar{\Ecal}_{\ns}(k)$ if and only if $|u-1|<1$.

Regardless of the reduction type of $E/\KK$, we let $z=-x/y$ denote the standard local parameter at the origin with respect to the chosen Weierstrass equation.  Let $B^\circ=\{P\in E(\KK) \mid \pi(P)=\pi(O)\}$ denote the ``kernel of reduction'', that is the set of points in $E(\KK)$ which reduce to the point $\pi(O)=(0:1:0)\in\bar{\Ecal}(k)\subset\PP^2(k)$.  Note that $B^\circ$ is a subgroup of $E(\KK)$.

\begin{lem}\label{DiscHomeoLem}
The map $z:B^\circ \to B^\circ(0,1)$ is a bijection of $B^\circ$ with the open unit disc $B^\circ(0,1)$ of $\KK$, and $|z(P\pm Q)|=|z(P)\pm z(Q)|$ for all $P,Q\in B^\circ$. 
\end{lem}

Before we get to the proof of this lemma, define 
\begin{equation}\label{MetricOnE}
d(P,Q) = 
\begin{cases}
|z(P-Q)| & \text{ if } P-Q\in B^\circ \\
1 & \text{ if } P-Q\not\in B^\circ
\end{cases}
\end{equation}
for any two points $P,Q\in E(\KK)$.  This definition does not depend on our choice of Weierstrass equation.  It is easy to check using Lemma \ref{DiscHomeoLem} that $d(P,Q)$ defines a non-archimedean metric on $E(\KK)$; we take the resulting metric topology as the definition of the topology on $E(\KK)$.

\begin{proof}[Proof of Lemma \ref{DiscHomeoLem}]
Let $F(Z_1,Z_2)\in\OO[[Z_1,Z_2]]$ be the formal group law associated to the integral Weierstrass equation, with inverse $i(Z)\in\OO[[Z]]$, and let $w(Z)\in\OO[[Z]]$ be the expansion of the function $z\mapsto w=-1/y$ on $E(\KK)$ at $z=0$.  These are formal power series with integral coefficients which converge when the variables $Z, Z_1, Z_2$ take values $z,z_1,z_1\in B^\circ(0,1)$; see for example \cite{SilvermanI} $\S$IV.1 for their definitions and basic properties.  

Given $z\in B^\circ(0,1)$, the point $\sigma(z):=(-z:1:-w(z))\in\PP^2(\KK)$ satisfies the Weierstrass equation and thus constitutes a point in $E(\KK)$.  Since $w(Z)= Z^3+u(Z)$ where all terms of $u(Z)\in\OO[[Z]]$ have degree at least $4$, we see that $|w(z)|<1$, and thus $\sigma(z)$ reduces to $(0:1:0)\in\PP^2(k)$; in other words $\sigma(z)\in B^\circ$.  It follows that the map $\sigma:B^\circ(0,1)\to B^\circ$ is the inverse of $z:B^\circ\to B^\circ(0,1)$, and thus both maps are bijections.

In fact the maps $\sigma:\hat{E}(\MM) \to B^\circ$ and $z:B^\circ \to \hat{E}(\MM)$ are group isomorphisms, where $\hat{E}(\MM)$ denotes the group which is equal to the set $B^\circ(0,1)$ endowed with the group law $(z_1,z_2)\mapsto F(z_1,z_2)$, with inverse $z\mapsto i(z)$.  Since $F(Z_1,Z_2)= Z_1+Z_2+G(Z_1,Z_2)$ where all terms of $G(Z_1,Z_2)\in\OO[[Z_1,Z_2]]$ have degree at least $2$, we have $|z(P+Q)|=|F(z(P),z(Q))|=|z(P)+z(Q)|$.  Moreover $i(Z)= -Z +j(Z)$ where all terms of $j(Z)\in\OO[[Z]]$ have degree at least $2$, and thus $|z(-P)|=|i(z(P))|=|z(P)|$.
\end{proof}

\subsection{The Berkovich unit disc}\label{BerkDiscSection}

In this section we will pause briefly from the discussion of elliptic curves and define the Berkovich unit disc $\Bsf(0,1)$ over $\KK$, the {\em open} Berkovich unit disc $\Bsf^\circ(0,1)$, summarize their basic properties, and describe their canonical path metric $\rho(x,y)$.  Our treatment of $\Bsf(0,1)$ follows Baker-Rumely \cite{BakerRumely2} Ch. 1.

Let $\KK\langle T\rangle = \{ f(T)=\sum_{\ell\geq0}a_\ell T^\ell \mid a_\ell\in\KK, |a_\ell|\to0 \}$ be the ring of formal power series converging in the unit disc $B(0,1)=\{a\in\KK \mid |a|\leq1 \}$ of $\KK$.  A multiplicative seminorm $[\cdot]$ on $\KK\langle T\rangle$ is a nonnegative real-valued function on $\KK\langle T\rangle$ which satisfies the axioms $[0]=0$, $[1]=1$, $[f(T)+g(T)]\leq[f(T)]+[g(T)]$, and $[f(T)g(T)]=[f(T)][g(T)]$ for all $f(T),g(T)\in \KK\langle T\rangle$.  Note that, unlike a true norm in the functional-analytic sense, it is not assumed that $[\cdot]$ is non-vanishing on nonzero elements of $\KK\langle T\rangle$.  A multiplicative seminorm $[\cdot]$ is said to be bounded if there exists a constant $C>0$ such that $[f(T)]\leq C\max_{\ell\geq0}|a_\ell|$ for all $f(T)=\sum_{\ell\geq0}a_\ell T^\ell\in\KK\langle T\rangle$.

The Berkovich unit disc $\Bsf(0,1)$ is defined to be the set of bounded multiplicative seminorms $[\cdot]$ on the ring $\KK\langle T\rangle$ which extend the absolute value $|\cdot|$ on $\KK$.  There exists a natural compact, Hausdorff topology on $\Bsf(0,1)$; it is defined as the weakest topology such that those subsets of the form
\begin{equation}\label{SubBase}
\{ x\in\Bsf(0,1) \mid \alpha<[f(T)]_x<\beta \} \hskip1cm (\alpha, \beta\in\RR, f(T)\in\KK\langle T\rangle)
\end{equation} 
are open, where here and throughout we denote by $[\cdot]_x$ the seminorm corresponding to the point $x\in\Bsf(0,1)$.  In other words, the collection of sets of the form $(\ref{SubBase})$ are a sub-base for the topology on $\Bsf(0,1)$; see \cite{Berkovich} $\S$~1.2 and \cite{BakerRumely2} $\S$~1.1.

To see some examples of elements of the Berkovich unit disc $\Bsf(0,1)$, note that each element $a$ in the ordinary unit disc $B(0,1)$ gives rise to the evaluation seminorm $[f(T)]_a=|f(a)|$, and this defines a dense embedding $B(0,1)\hookrightarrow\Bsf(0,1)$.  More generally, each closed disc $B(a,r)=\{z\in\KK \mid |z-a|\leq r \}$ ($a\in B(0,1)$, $r\in\{0\}\cup|\KK^\times|$, $0\leq r\leq1$) defines a point $\zeta_{a,r}$ in $\Bsf(0,1)$ corresponding to the sup norm $[f(T)]_{\zeta_{a,r}}=\sup_{z\in B(a,r)}|f(z)|$.  Thus we identify the ``classical point'' $a\in B(0,1)$ with the ``Berkovich point'' $\zeta_{a,0}\in \Bsf(0,1)$ under this notation.  It may happen (depending on the field $\KK$) that the points $\zeta_{a,r}$ account for all elements of $\Bsf(0,1)$, but this is generally not the case.  In fact, Berkovich (\cite{Berkovich} $\S$1.4.4) has given a classification of all points of $\Bsf(0,1)$ into four {\em types}, of which the points $\zeta_{a,0}$ (for $a\in B(0,1)$) compose type 1, and the points $\zeta_{a,r}$ (for $a\in B(0,1)$, $r\in|\KK^\times|$, $0< r\leq1$) compose type 2.  Loosely speaking, the points of type 3 and 4 ``fill in the holes'' of the set $\{\zeta_{a,r}\}$ to create a compact, path-connected space $\Bsf(0,1)$.  

The path-connectedness of $\Bsf(0,1)$ can be easily visualized.  First observe that $\Bsf(0,1)$ carries a natural partial order under which $x\leq y$ if and only if $[f(T)]_x\leq[f(T)]_y$ for all $f(T)\in\KK\langle T\rangle$.  Thus if $\zeta_{a,r}$ and $\zeta_{a',r'}$ are points of type 1 or 2, then $\zeta_{a,r}\leq\zeta_{a',r'}$ if and only if $B(a,r)\subseteq B(a',r')$.  It follows from Berkovich's classification that the point ${\zeta_{0,1}}$ corresponding to the sup norm on the unit disc $B(0,1)$ itself is the unique maximal point in $\Bsf(0,1)$; this point is commonly known as the {\em Gauss point} of $\Bsf(0,1)$.  On the other hand, each type 1 point $a=\zeta_{a,0}$ in $B(0,1)$ is minimal.  Given two points $x,y\in\Bsf(0,1)$ with $x\leq y$, denote by $[x,y]$ the set of all $z\in\Bsf(0,1)$ satisfying $x\leq z\leq y$.  It is a straightforward exercise to show that $[x,y]$ is homeomorphic to a closed subinterval of the real line.  More generally, given any two $x,y\in\Bsf(0,1)$ (perhaps with $x\not\leq y$), there exists a unique least-upper-bound $x\vee y$ of $x$ and $y$; we then define $[x,y]=[x,x\vee y]\cup[y,x\vee y]$, which is again homeomorphic to a closed real interval; we will call such sets $[x,y]$ {\em line segments} in $\Bsf(0,1)$.  We remark that $[x,y]=[y,x]$ for all $x,y\in\Bsf(0,1)$; that is, the notation does not imply an order relationship between $x$ and $y$.

\medskip

In addition to the above topological considerations, there is a natural way to assign a notion of length to each line segment in $\Bsf(0,1)$.  In order to do this we first consider the diameter function $\diam:\Bsf(0,1)\to[0,1]$ defined by $\diam(x)=\inf_{a\in B(0,1)}[T-a]_x$; thus $\diam(\zeta_{a,r})=r$ for points of type 1 and 2.  Then given two points $x,y\in\Bsf(0,1)\setminus B(0,1)$, define the length of the line segment $[x,y]$ by
\begin{equation*}
\rho(x,y) = \log\bigg(\frac{\diam(x\vee y)^2}{\diam(x) \diam(y)}\bigg).
\end{equation*}
This defines a path metric on $\Bsf(0,1)\setminus B(0,1)$ under which the points of $B(0,1)$ can be viewed as being at infinite distance from all other points.  Moreover, any line segment $[x,y]$ in $\Bsf(0,1)\setminus B(0,1)$ is isometric to a real interval of length $\rho(x,y)$.

We note for future use the following fact: let $x,y\in B(0,1)$ be two distinct type 1 points.  Then it is easy to see that $x\vee y=\zeta_{x,|x-y|}$, and therefore the distance from the Gauss point $\zeta_{0,1}$ to the point $x\vee y$ is 
\begin{equation}\label{LUBFormula}
\rho(\zeta_{0,1},x\vee y) = -\log|x-y|.
\end{equation}

\medskip

For our purposes it will be more useful to work with the {\em open} Berkovich unit disc $\Bsf^\circ(0,1)$, which is a certain proper subset of $\Bsf(0,1)$.  To define $\Bsf^\circ(0,1)$, let $B^\circ(0,1)=\{a\in\KK \mid |a|<1 \}$ be the open unit disc in $B(0,1)$.  Given $a\in B(0,1)$, it is clear that $a\in B^\circ(0,1)$ if and only if $0\vee a\neq\zeta_{0,1}$ in $\Bsf(0,1)$.  Motivated by this we define the open Berkovich unit disc by
\begin{equation*}
\Bsf^\circ(0,1) = \{x\in \Bsf(0,1) \mid 0\vee x\neq\zeta_{0,1}\}.
\end{equation*}
Thus $\Bsf^\circ(0,1)\cap B(0,1)=B^\circ(0,1)$.  Like $\Bsf(0,1)$, the set $\Bsf^\circ(0,1)$ is path connected, but unlike $\Bsf(0,1)$, $\Bsf^\circ(0,1)$ has no maximal element.  On the other hand the set $\Bsf^\circ(0,1)\cup\{\zeta_{0,1}\}$ obtained by adjoining the Gauss point is path connected, with $\zeta_{0,1}$ as its unique maximal element.

\subsection{The space $\Esf$ in the good reduction case}\label{BerkSectionGR}

Returning to our elliptic curve $E/\KK$, suppose that $|j_E|\leq1$.  For each $\alpha\in\bar{\Ecal}(k)$ let $B_\alpha^\circ=\pi^{-1}(\alpha)=\{P\in E(\KK) \mid \pi(P)=\alpha \}$ be the set of points reducing to $\alpha$, and fix a point $P_\alpha\in B_\alpha^\circ$.  It follows from Lemma~\ref{DiscHomeoLem} that the map
\begin{equation}\label{AlphaHomeo}
\nu_\alpha:B_\alpha^\circ \to B^\circ(0,1) \hskip1cm P\mapsto z(P-P_\alpha)
\end{equation}
is a homeomorphism of $B_\alpha^\circ$ with the open unit disc $B^\circ(0,1)$ of $\KK$.  We thus have a decomposition of $E(\KK)$ into a disjoint union
\begin{equation}\label{EKDisjointUnionGR}
E(\KK) = \amalg_{\alpha\in \bar{\Ecal}(k)} B_\alpha^\circ
\end{equation}
of subsets which are homeomorphic to open discs.
 
For each $\alpha\in \bar{\Ecal}(k)$, let $\Bsf_\alpha^\circ$ denote the open Berkovich unit disc associated to the open unit disc $B_\alpha^\circ$, and let $\Bsf_\alpha^\circ\cup\{\zeta_\alpha\}$ denote the union of $\Bsf_\alpha^\circ$ with its associated Gauss point $\zeta_\alpha$, as described in $\S$\ref{BerkDiscSection}.  The Berkovich analytic space $\Esf$ is the union of the $\Bsf_\alpha^\circ\cup\{\zeta_\alpha\}$ over all $\alpha\in \bar{\Ecal}(k)$, where the points $\zeta_\alpha$ are identified to form a single point $\zeta\in\Esf$.  The set $\Sigma=\{\zeta\}$ is the skeleton of $\Esf$, and we therefore have a decomposition 
\begin{equation}\label{ConnCompPar1}
\Esf\setminus\Sigma = \amalg_{\alpha\in \bar{\Ecal}(k)} \Bsf_\alpha^\circ
\end{equation}
of $\Esf\setminus\Sigma$ into its connected components $\Bsf_\alpha^\circ$.  The constant map $r_\Sigma:\Esf\to\Sigma$ defines a strong deformation retraction of $\Esf$ onto $\Sigma$; in particular, $\Esf$ is simply connected.

\subsection{The space $\Esf$ in the  multiplicative reduction case}\label{BerkSectionMR}

Suppose now that $|j_E|>1$.  Let $u:E(\KK)\stackrel{\sim}{\to}\KK^\times/q^\ZZ$ be the inverse of the Tate isomorphism $\phi$ discussed in $\S$\ref{BerkSection}.  Then
\begin{equation}\label{RetractionMap}
r:E(\KK)\to\Sigma \hskip1cm r(P)=-\log|u(P)|
\end{equation}
is a group homomorphism, where $\Sigma=\RR/(\log|j_E|)\ZZ$ denotes the circle group of circumference $\log|j_E|>0$.  As explained in $\S$~\ref{BerkSection}, the kernel of the map $r$ is precisely the set $E_0(\KK)$ of points in $E(\KK)$ with nonsingular reduction.  More generally, let $\Sigma_r\subseteq\Sigma$ denote the image of $r$, and for each $s\in\Sigma_r$ let $E_s(\KK)=r^{-1}(s)$ denote its $r$-preimage.  

For each $s\in\Sigma_r$, select a point $P_{s,0}\in E_s(\KK)$.  Then for each pair $(s,\alpha)\in\Sigma_r\times\bar{\Ecal}_{\ns}(k)$ define a subset
\begin{equation*}
B_{s,\alpha}^\circ=\{P\in E_s(\KK) \mid \pi_0(P-P_{s,0})=\alpha\}
\end{equation*}
of $E_s(\KK)$.  We have decomposed $E(\KK)$ into a disjoint union
\begin{equation}\label{EKDisjointUnionMR}
E(\KK) = \amalg_{(s,\alpha)\in \Sigma_r\times\bar{\Ecal}_{\ns}(k)} B_{s,\alpha}^\circ.
\end{equation}
Note that $P_{s,0}\in B_{s,0}^\circ$.  More generally, for each pair $(s,\alpha)$ with $\alpha\neq0$, select a point $P_{s,\alpha}\in B_{s,\alpha}^\circ$.  It follows from Lemma~\ref{DiscHomeoLem} that the map
\begin{equation}\label{AlphaHomeoMR}
\nu_{s,\alpha}:B_{s,\alpha}^\circ \to B^\circ(0,1) \hskip1cm P\mapsto z(P-P_{s,\alpha})
\end{equation}
is a homeomorphism of $B_{s,\alpha}^\circ$ with the open unit disc $B^\circ(0,1)$ of $\KK$.

For each pair $(s,\alpha)\in \Sigma_r\times\bar{\Ecal}_{\ns}(k)$, let $\Bsf_{s,\alpha}^\circ$ denote the open Berkovich unit disc associated to the open unit disc $B_{s,\alpha}^\circ$, and let $\Bsf_{s,\alpha}^\circ\cup\{\zeta_{s,\alpha}\}$ denote the union of $\Bsf_{s,\alpha}^\circ$ with its associated Gauss point $\zeta_{s,\alpha}$, as described in $\S$\ref{BerkDiscSection}.  The Berkovich analytic space $\Esf$ is the union of the circle $\Sigma$ with the sets $\Bsf_{s,\alpha}^\circ\cup\{\zeta_{s,\alpha}\}$ over all $(s,\alpha)\in \Sigma_r\times\bar{\Ecal}_{\ns}(k)$, where for each $s\in\Sigma_r$ the points $\zeta_{s,\alpha}$ are identified with each other and with the point $s\in\Sigma_r$ to form a single point $s$ of $\Esf$.  Thus $\Esf$ can be visualized as a collection of open Berkovich discs $\Bsf_{s,\alpha}^\circ$ glued together along the skeleton $\Sigma$.  We have a decomposition 
\begin{equation}\label{ConnCompPar2}
\Esf\setminus\Sigma = \amalg_{(s,\alpha)\in \Sigma_r\times\bar{\Ecal}_{\ns}(k)} \Bsf_{s,\alpha}^\circ
\end{equation}
of $\Esf\setminus\Sigma$ into its connected components $\Bsf_{s,\alpha}^\circ$.  The homomorphism $r:E(\KK)\to\Sigma$ defined in $(\ref{RetractionMap})$ extends to a strong deformation retraction $r_\Sigma:\Esf\to\Sigma$; in particular, the fundamental group of $\Esf$ is isomorphic to $\pi_1(\Sigma)\simeq\ZZ$.

Note that the parametrization of the sets $\Bsf_{s,\alpha}^\circ$ by $\Sigma_r\times\bar{\Ecal}_{\ns}(k)$ is non-canonical, due to the arbitrary choice of each point $P_{s,\alpha}$ in $B_{s,\alpha}^\circ$.

\subsection{The path metric}

The path metric $\rho(x,y)$ on $\Esf\setminus E(\KK)$ is the unique metric which restricts to the canonical path metric on each connected component $\Bsf^\circ$ of $\Esf\setminus E(\KK)$, as described in $\S$\ref{BerkDiscSection}, and which in the multiplicative reduction case restricts to the usual path metric on the circle $\Sigma\simeq\RR/(\log|j_E|)\ZZ$.  Under this metric, each point of $E(\KK)$ can be viewed as being at ``infinite distance'' from all other points in $\Esf$.  

We caution that the metric topology on $\Esf\setminus E(\KK)$ is not the same as the subspace topology induced by the Hausdorff topology on $\Esf$.  However, the metric $\rho(x,y)$ and the topology on $\Esf$ are related to one another; we will describe this relationship in $\S$\ref{BerkTopology}.

\subsection{The N\'eron function}\label{NeronFunction}

Let $\lambda:E(\KK)\setminus\{O\}\to\RR$ be the N\'eron function as defined and normalized in \cite{SilvermanII}, $\S$VI.1.  There exists a natural extension of $\lambda$ to a continuous function on $\Esf\setminus\{O\}$ which is closely related to the analytic structure of the space $\Esf$, its skeleton $\Sigma$, and the path metric $\rho(x,y)$.

In order to describe this extension, we begin by defining two continuous functions 
\begin{equation*}
i:\Esf\times\Esf\to\RR\cup\{+\infty\} \hskip1cm j:\Esf\times\Esf\to\RR.
\end{equation*}
Letting $\diag(E(\KK))$ denote the diagonal of $E(\KK)\times E(\KK)$ in $\Esf\times\Esf$, we first declare that $i(x,y)=+\infty$ if $(x,y)\in\diag(E(\KK))$.  Given $(x,y)\in\Esf\times\Esf\setminus\diag(E(\KK))$, define $i(x,y)=0$ if $x$ and $y$ are not both elements of the same connected component $\Bsf^\circ$ of $\Esf\setminus\Sigma$, as in one of the decompositions $(\ref{ConnCompPar1})$ or $(\ref{ConnCompPar2})$.  On the other hand, if both $x$ and $y$ are elements of the same connected component $\Bsf^\circ$, let $x\vee y\in \Esf\setminus E(\KK)$ be their least upper bound with respect to the partial order on $\Bsf^\circ$, and define
\begin{equation}\label{IDef}
i(x,y) = \rho(r_\Sigma(x\vee y),x\vee y).
\end{equation}
In other words, $i(x,y)$ is the distance between $x\vee y$ and its nearest point on the skeleton $\Sigma$.  Note that $i(x,y)$ is nonnegative, finite outside of $\diag(E(\KK))$, and positive if and only if $x$ and $y$ lie in the same connected component of $\Esf\setminus\Sigma$.  In particular, if $P,Q\in E(\KK)$ then
\begin{equation}\label{ReductionIFunction}
i(P,Q) > 0 \hskip5mm\Leftrightarrow \hskip5mm\pi(P)=\pi(Q) \, \text{ and }\, r_\Sigma(P)=r_\Sigma(Q).
\end{equation}
The number $i(P,Q)$ is closely related to the intersection multiplicity of the two points $P,Q\in E(\KK)$ in the integral model $\Ecal/\OO$; for example, assuming $E/\KK$ has good reduction, $(\ref{ReductionIFunction})$ shows that $i(P,Q)=0$ if and only if $P$ and $Q$ do not meet in $\Ecal$.

Finally we point out that on $E(\KK)$ the function $i(P,Q)$ is related to the metric $(\ref{MetricOnE})$ by $i(P,Q)=-\log d(P,Q)$; this is easy to see using the definitions $(\ref{MetricOnE})$ and $(\ref{IDef})$ and the rule $(\ref{LUBFormula})$.

We now define $j(x,y)$.  Let $j(x,y)$ be identically zero if $E/\KK$ has good reduction.  If $E/\KK$ has multiplicative reduction, let $\ell=\log|j_E|=\log|1/q|>0$, let $r_\Sigma:\Esf\to\Sigma\simeq\RR/\ell\ZZ$ denote the retraction map, and define
\begin{equation}\label{JDefinition}
j(x,y) = \frac{\ell}{2}\Phi\bigg(\frac{r_\Sigma(x)-r_\Sigma(y)}{\ell}\bigg),
\end{equation}
where $\Phi(t)=(t-[t])^2-(t-[t])+1/6$ is the second periodic Bernoulli polynomial.  Regardless of the reduction type of $E/\KK$ we see that $j(x,y)$ is bounded and symmetric, it factors through the retraction $r_\Sigma:\Esf\to\Sigma$ onto the skeleton, and $j(x,y)= \frac{1}{12}\log^+|j_E|$ if $r_\Sigma(x)=r_\Sigma(y)$.

The function $j(x,y)$ is quite natural when viewed in the context of the Laplacian on $\Sigma$ as a metrized graph.  In the language of Baker-Rumely (\cite{BakerRumely3} \S1.8), $j(x,y)$ is the unique normalized Arakelov-Green's function on the circle $\Sigma$ with respect to the uniform probability measure.

\begin{prop}\label{LocalHeightProp}
Let $P$ and $Q$ be distinct points in $E(\KK)$.  Then
\begin{equation}\label{ChinburgRumely}
\lambda(P-Q) = i(P,Q) + j(P,Q).
\end{equation}
\end{prop}

In view of this proposition we define a function of two variables $\lambda:\Esf\times\Esf\to\RR\cup\{+\infty\}$ by $\lambda(x,y)=i(x,y) + j(x,y)$.  Thus $\lambda(x,y)=+\infty$ if and only if $(x,y)\in \diag(E(\KK))$, and Proposition~\ref{LocalHeightProp} states that $\lambda(P,Q) = \lambda(P-Q)$ for distinct points $P,Q\in E(\KK)$.  The decomposition $(\ref{ChinburgRumely})$ is analogous to the one introduced by Chinburg-Rumely \cite{ChinburgRumely} in the context of the reduction graph associated to an arithmetic surface over a discrete valuation ring; from the point of view of Berkovich analytic spaces this is closely related to the skeleton $\Sigma$ of $\Esf$.

\begin{proof}[Proof of Proposition~\ref{LocalHeightProp}]
First suppose that $E/\KK$ has good reduction.  Then 
\begin{equation}\label{LocalHeightGR}
\lambda(P-Q)=\frac{1}{2}\log^+|x(P-Q)|=\log^+|z(P-Q)^{-1}|,
\end{equation}
where $z=-x/y$ is the local parameter at the origin associated to the integral Weierstrass equation $(\ref{WeierstrassEq})$ in the variables $x$ and $y$.  The first identity in $(\ref{LocalHeightGR})$ is proved in \cite{SilvermanII} Thm. VI.4.1.  

To see the second identity in $(\ref{LocalHeightGR})$, let us abbreviate $x=x(P-Q)$, $y=y(P-Q)$, and $z=z(P-Q)$.  First suppose that  $|x|>1$.  Since the Weierstrass equation $(\ref{WeierstrassEq})$ has coefficients in $\OO$, it follows from the ultrametric inequality that $|y|^2=|x|^3$, whereby $|z^{-1}|=|y|/|x|=|x|^{1/2}>1$.  This proves the second identity in $(\ref{LocalHeightGR})$ in this case.

Now consider the case $|x|\leq1$.  It suffices to show that $|z^{-1}|\leq1$, which will imply that both $\log^+|x|$ and $\log^+|z^{-1}|$ vanish, completing the proof of the second identity in $(\ref{LocalHeightGR})$.  Suppose on the contrary that $|z^{-1}|>1$.  Then $|z|<1$, and therefore $|w(z)|<1$, where $w(Z)\in\OO[[Z]]$ is the power series expansion of the function $z\mapsto w=-1/y$ on $E(\KK)$ at $z=0$, as discussed in the proof of Lemma \ref{DiscHomeoLem}.  Thus $|y|^{-1}=|w|<1$, whereby $|y|>1$.  But since $|x|\leq1$, it follows from the ultrametric inequality that the left-hand-side of the Weierstrass equation $(\ref{WeierstrassEq})$ has absolute value $|y|^2>1$, while the right-hand-side of $(\ref{WeierstrassEq})$ has absolute value $\leq1$.  The contradiction proves that $|z^{-1}|\leq1$, as desired, and completes the proof of the second identity in $(\ref{LocalHeightGR})$.

We will now use $(\ref{LocalHeightGR})$ to prove $(\ref{ChinburgRumely})$.  By definition $j(P,Q)=0$, and $i(P,Q)=0$ if $P$ and $Q$ are not both elements of the same disc $B_\alpha^\circ$ as defined in $\S$\ref{BerkSectionGR}.  On the other hand, suppose $P$ and $Q$ are elements of $B_\alpha^\circ$, let $\nu_\alpha:B_\alpha^\circ\to B^\circ(0,1)$ denote the homeomorphism $(\ref{AlphaHomeo})$, and let $\nu_\alpha(P)\vee \nu_\alpha(Q)$ denote their least upper bound in $\Bsf^\circ(0,1)$.  Then
\begin{equation}\label{ICalcGR}
\begin{split}
i(P,Q) & = \rho(r_\Sigma(\nu_\alpha(P)\vee \nu_\alpha(Q)),\nu_\alpha(P)\vee \nu_\alpha(Q)) \\
	& = -\log|\nu_\alpha(P)-\nu_\alpha(Q)| \hskip1cm \text{ by } (\ref{LUBFormula}) \\
	& = -\log|z(P-P_\alpha)-z(Q-P_\alpha)| \\
	& = -\log|z(P-Q)|>0 \hskip1cm \text{ by Lemma \ref{DiscHomeoLem}}.
\end{split}
\end{equation}
It follows that $\lambda(P-Q)=\log^+|z(P-Q)^{-1}|=i(P,Q)=i(P,Q) + j(P,Q)$, completing the proof of $(\ref{ChinburgRumely})$ in the good reduction case.

We now turn to the case that $E/\KK$ has multiplicative reduction.  Let $\phi:\KK^\times/q^\ZZ\stackrel{\sim}{\to}E(\KK)$ denote the Tate isomorphism.  By periodicity we may select an element $u\in\KK^\times\setminus\{1\}$ such that $\phi(u)=P-Q$ and $|q|<|u|\leq1$.  We are going to calculate the sum $i(P,Q)+j(P,Q)$ in three separate cases.

Case 1: $|q|<|u|<1$. Then $r_\Sigma(P)\neq r_\Sigma(Q)$ by $(\ref{RetractionMap})$, whereby $i(P,Q)=0$ by $(\ref{ReductionIFunction})$.  It follows from the definition $(\ref{RetractionMap})$ of the retraction map $r_\Sigma$ and the definition $(\ref{JDefinition})$ of $j(P,Q)$ that
\begin{equation}\label{IJIdentityMR}
i(P,Q)+j(P,Q) = j(P,Q)=\frac{\log|1/q|}{2}\Phi\bigg(\frac{\log|1/u|}{\log|1/q|}\bigg).
\end{equation}

Case 2: $|u|=1$ and $|u-1|=1$.  In this case $r_\Sigma(P)= r_\Sigma(Q)$ by $(\ref{RetractionMap})$, whereby $j(P,Q)=\frac{1}{12}\log|1/q|$.  As explained in $\S$~\ref{BerkSection}, the condition $|u-1|=1$ means that $P-Q$ does not reduce to the identity element of $\bar{\Ecal}_{\ns}(k)$, whereby $\pi(P)\neq\pi(Q)$.  Thus $i(P,Q)=0$ by $(\ref{ReductionIFunction})$, and so we have
\begin{equation}\label{IJIdentityMR2}
i(P,Q)+j(P,Q) = j(P,Q)=\frac{1}{12}\log|1/q|.
\end{equation}

Case 3: $|u|=1$ and $|u-1|<1$.  We again have $r_\Sigma(P)= r_\Sigma(Q)$ by $(\ref{RetractionMap})$, and thus $j(P,Q)=\frac{1}{12}\log|1/q|$.  Writing the point $P-Q$ in the affine coordinates $(x,y)\in\KK^2$ associated to the Weierstrass equation $(\ref{TateEq})$, and letting $z=-x/y$, we have  
\begin{equation}\label{ICalcMR3}
i(P,Q) = -\log|z|=-\log|u-1|.
\end{equation}
The first identity in $(\ref{ICalcMR3})$ follows from the exact same calculation $(\ref{ICalcGR})$ used in the good reduction case.  To see the second identity in $(\ref{ICalcMR3})$, note that the coordinates $x=X(u,q)$ and $y=Y(u,q)$ are defined via certain convergent series; see \cite{SilvermanII} Thm. V.3.1 (c).  Our assumptions that $|u|=1$ and $|u-1|<1$, along with the ultrametric inequality, imply that $|x|=|X(u,q)|=|u-1|^{-2}$ and $|y|=|Y(u,q)|=|u-1|^{-3}$, whereby $|z|=|x|/|y|=|u-1|$, and the second identity in $(\ref{ICalcMR3})$ follows.  We have shown that
\begin{equation}\label{IJIdentityMR3}
i(P,Q)+j(P,Q) =-\log|u-1|+\frac{1}{12}\log|1/q|.
\end{equation}

To complete the proof of $(\ref{ChinburgRumely})$ in the multiplicative reduction case we observe that the right-hand-sides of the identities $(\ref{IJIdentityMR})$, $(\ref{IJIdentityMR2})$, and $(\ref{IJIdentityMR3})$ coincide with the formulas for $\lambda(P-Q)$ given in \cite{SilvermanII} Thm. VI.4.2 (b).
\end{proof}

\subsection{Connected metrized subgraphs of $\Esf\setminus E(\KK)$ containing $\Sigma$}\label{BerkTopology}

Given a finite subset $S$ of $\Esf\setminus(\Sigma\cup E(\KK))$, there exists a unique smallest path-connected subset $\Gamma$ of $\Esf$ containing $\Sigma\cup S$.  Each set $\Gamma$ arising in this way inherits from the path metric $\rho(x,y)$ the structure of connected metrized graph, and the induced topology on each such subgraph $\Gamma$ is closely related to the topology on $\Esf$.  These metrized subgraphs $\Gamma$ of $\Esf$ will be the basic analytic objects of study in this paper, and so in this section we will give a detailed description of their properties.

First, for each point $b$ in $\Esf\setminus(\Sigma\cup E(\KK))$, let $[r_\Sigma(b),b]$ denote the line segment in $\Esf$ from $r_\Sigma(b)\in\Sigma$ to $b$, and recall from $\S$\ref{BerkDiscSection} that $[r_\Sigma(b),b]$ is isometric (with respect to the path metric $\rho(x,y)$ on $\Esf\setminus E(\KK)$) to a closed real interval of length $\rho(r_\Sigma(b),b)>0$.  Given a (possibly empty) finite subset $S$ of $\Esf\setminus(\Sigma\cup E(\KK))$, define a subset $\Gamma$ of $\Esf$ by
\begin{equation}\label{GammaUnion}
\Gamma = \Sigma \cup \bigcup_{b\in S}[r_\Sigma(b),b]
\end{equation}
The path metric $\rho(x,y)$ endows the set $\Gamma$ with the structure of a connected metrized graph, and in particular $\Gamma$ carries a compact Hausdorff topology.  Define a map $r_\Gamma:\Esf\to\Gamma$ as follows.  Set $r_\Gamma(x)=x$ if $x\in\Gamma$, but if $x\not\in\Gamma$ let $\Bsf^\circ\subset\Esf\setminus\Sigma$ denote the open Berkovich disc containing $x$, as in either $(\ref{ConnCompPar1})$ or $(\ref{ConnCompPar2})$.  Define $r_\Gamma(x)$ to be unique smallest element in $[r_\Sigma(x),x]$ (with respect to the partial order $\leq$ on $\Bsf^\circ$) which is contained in $\Gamma$.  We note for future use the following fact: if $x,y$ are two points in the same connected component $\Bsf^\circ$ of $\Esf\setminus\Sigma$, then
\begin{equation}\label{RetVeeRule}
r_\Gamma(x)\vee r_\Gamma(y)=r_\Gamma(x\vee y),
\end{equation}
where $\vee$ denotes the least upper bound with respect to the partial order $\leq$ on $\Bsf^\circ$.

Let $U$ be a subset of $\Gamma$ which is open with respect to the metrized graph topology.  Then the set $r_\Gamma^{-1}(U)$ is an open subset of $\Esf$, and the collection of such sets $r_\Gamma^{-1}(U)$, over all connected metrized subgraphs $\Gamma$ of $\Esf\setminus E(\KK)$ containing $\Sigma$ and all open subsets $U$ of $\Gamma$, form a base of open sets for the topology on $\Esf$.  Under this topology $\Esf$ is a path-connected, compact, Hausdorff space, the map $r_\Gamma:\Esf\to\Gamma$ is a strong deformation retraction, and $E(\KK)$ is a dense subspace of $\Esf$ (see \cite{Berkovich} Ch. 4 and \cite{Baker1} $\S$~5).

\medskip

Note that we do not claim that the sets $[r_\Sigma(b),b]$ ($b\in S$) in the union $(\ref{GammaUnion})$ are pairwise disjoint; any one of them may be contained in another, or any two of them may coincide on some initial sub-segment.  On the other hand, we may write any such connected subgraph $\Gamma$ of $\Esf\setminus E(\KK)$ containing $\Sigma$ as a disjoint union
\begin{equation}\label{GammaDisjointUnion}
\Gamma = \Sigma \amalg (a_1,b_1] \amalg \dots \amalg (a_M,b_M],
\end{equation}
where each $(a_m,b_m]$ is a subset of $\Esf$ which is isometric to a half-open real interval of length $\rho(a_m,b_m)>0$.  Recall that the length of the skeleton is $\ell(\Sigma)=\log^+|j_E|$, and thus the length of $\Gamma$ is $\ell(\Gamma)=\log^+|j_E|+\sum_{m=1}^{M}\rho(a_m,b_m)$.  Thus $\ell(\Gamma)=0$ if and only if $E/\KK$ has good reduction and $\Gamma=\Sigma$.  

\medskip

Assuming for the rest of this section that $\ell(\Gamma)>0$, denote by $m_\Gamma$ the uniform measure on $\Gamma$ with respect to the metric $\rho(x,y)$.  Thus if $F:\Gamma\to\RR$ is $m_\Gamma$-integrable we have
\begin{equation*}
\int_\Gamma F(x)dm_\Gamma(x) = \int_{\Sigma}F(x)dm_\Sigma(x) +\sum_{m=1}^{M}\int_{(a_m,b_m]}F(x)dx,
\end{equation*}
where the first term on the right-hand-side is present only in the multiplicative reduction case, and $dx$ denotes the uniform (Lebesgue) measure on $(a_m,b_m]$ for each $1\leq m\leq M$.  Denote by $L^2(\Gamma)$ the space of square $m_\Gamma$-integrable real-valued functions on $\Gamma$, with norm $\|F\|_{L^2(\Gamma)}=\{\int_\Gamma|F|^2dm_\Gamma\}^{1/2}$.  

\subsection{The discreteness of torsion}

We conclude this preliminary section with the following lemma.  We will only need it for a global application in $\S$\ref{SIntSection}, but as it is purely local we include it here.

\begin{lem}\label{TorsionDiscrete}
The torsion subgroup $E(\KK)_\tor$ is discrete in $E(\KK)$.
\end{lem}

This is well-known, at least in the characteristic-zero case where it follows at once from properties of the formal logarithm.  The following proof holds in arbitrary characteristic and was pointed out to us by Felipe Voloch. 

\begin{proof}[Proof of Lemma \ref{TorsionDiscrete}]
By translation it suffices to show that the origin $O\in E(\KK)$ has a torsion-free open neighborhood $U$.  Let $B^\circ$ be the kernel of reduction associated to an integral Weierstrass equation for $E$, let $z:B^\circ \to\hat{E}(\MM)$ be the isomorphism of Lemma \ref{DiscHomeoLem}, where $\hat{E}(\MM)$ denotes the formal group over $\MM$ with formal group law $F(Z_1,Z_2)\in\OO[[Z_1,Z_2]]$ associated to the Weierstrass equation.  

If the characteristic of the residue field $k$ is zero then $\hat{E}(\MM)$ is torsion free (\cite{SilvermanI} IV.3.2(b)), and so we may take $U=B^\circ$ itself.  Thus we may now assume $\mathrm{char}(k)=p\neq0$; in this case $\hat{E}(\MM)$ has only $p$-torsion (\cite{SilvermanI} IV.3.2(b)).  The multiplication-by-$p$ map $[p]:\hat{E}(\MM)\to\hat{E}(\MM)$ is given by $[p](z)=pz+\dots$ if $\mathrm{char}(\KK)=0$, and by $[p](z)=az^{p}+\dots$ for some $a\in\OO$ if $\mathrm{char}(\KK)=p$.  In either case it is easy to see that $[p](B(0,R))\subseteq B(0,R)$ for all sufficiently small $R>0$, where $B(0,R)=\{z\in\MM\mid |z|\leq R\}$.  Since the set $E[p]$ of $p$-torsion on $E(\KK)$ is finite, if we take $R$ small enough we may assume that $[p](B(0,R))\subseteq B(0,R)$ and that $B(0,R)$ contains no $p$-torsion other than $0$.

Suppose that $B(0,R)$ contains a point $z$ of exact order $p^r$ ($r\geq2$).  Then $[p^{r-1}](z)$ is nonzero, has order $p$, and is in $B(0,R)$ since $[p](B(0,R))\subseteq B(0,R)$.  This contradicts the fact that $B(0,R)$ contains no $p$-torsion other than $0$.  Thus $B(0,R)$ contains no nonzero $p$-power torsion, so it is torsion-free.  Let $U$ be the pullback of $B(0,R)$ under the isomorphism $z:B^\circ \to\hat{E}(\MM)$.
\end{proof}

%%%%%%%%%%%%%%%%%%%%%%%%%%%%%%%%%%%%%%%%%%%%%%%%%%%%%%%%%%%%
%%%%%%%%%%%%%%%%%%%%%%%%%%%%%%%%%%%%%%%%%%%%%%%%%%%%%%%%%%%%
%%%%%%%%%%%%%%%%%%%%%%%%%%%%%%%%%%%%%%%%%%%%%%%%%%%%%%%%%%%%

\section{The Main Local Inequality}\label{MainLocalSection}

\subsection{Overview}

We continue with the notation of the last section.  In this section we will define the canonical measure $\mu$ on $\Esf$, and we will prove the main local result of this paper, which is a bound on the error term when testing the $\mu$-equidistribution of a set $Z$ against a certain class of test functions on $\Esf$.

\subsection{The space $\Scal(\Esf,\RR)$ of test functions}

Let $\Ccal(\Esf,\RR)$ denote the Banach space of continuous real-valued functions on $\Esf$, equipped with the supremum norm.  We will now define a dense subspace $\Scal(\Esf,\RR)$ of $\Ccal(\Esf,\RR)$ which will serve as our space of test functions.

First, let $\Gamma$ be a connected metrized subgraph of $\Esf\setminus E(\KK)$ containing $\Sigma$, and let $F:\Gamma\to\RR$ be a function.  Given $x\in(r_\Sigma(b),b)$, where $\Gamma$ is written as in $(\ref{GammaUnion})$, define the derivative $F'(x)$ of $F$ at $x$ via the usual difference quotient limit, assuming it exists, by identifying $(r_\Sigma(b),b)$ with an open real interval.  If $E/\KK$ has multiplicative reduction and $x\in\Sigma$, define $F'(x)$ similarly by identifying $\Sigma\simeq\RR/(\log|j_E|)\ZZ$.  Thus the derivative $F'(x)$, assuming it exists, has an unambiguous meaning for $m_\Gamma$-almost all  $x\in\Gamma$.  Define $\Scal_\Gamma(\Esf,\RR)$ to be the space of functions $F:\Esf\to\RR$ satisfying the following conditions:\begin{itemize}
	\item $F$ factors through the retraction map $r_\Gamma:\Esf\to\Gamma$; thus $F=F\circ r_\Gamma$.
	\item $F$ is continuous on $\Gamma$ with respect to its metrized graph topology.
	\item The derivative $F'$ exists $m_\Gamma$-almost everywhere on $\Gamma$ and $F'\in L^2(\Gamma)$.
\end{itemize}
We call $\Scal_\Gamma(\Esf,\RR)$ the space of test functions associated to $\Gamma$.

Now define $\Scal(\Esf,\RR)=\cup_\Gamma\Scal_\Gamma(\Esf,\RR)$, the union over all connected metrized subgraphs $\Gamma$ of $\Esf\setminus E(\KK)$ containing $\Sigma$.  It follows from the definition of the topology on $\Esf$ that $\Scal(\Esf,\RR)\subset\Ccal(\Esf,\RR)$, and it is straightforward to show using the Stone-Weierstrass theorem that $\Scal(\Esf,\RR)$ is dense in $\Ccal(\Esf,\RR)$.

\subsection{The canonical measure and the local discrepancy}\label{CanMeasLocDisc}

The canonical measure $\mu$ on $\Esf$ is the unique positive unit Borel measure which is supported on the skeleton $\Sigma$, and which when restricted to $\Sigma$ is given by
\begin{equation*}
\mu = 
\begin{cases}
\text{ the Dirac measure at the point $\Sigma$ if $|j_E|\leq1$} \\
\text{ the normalized uniform measure $(\log|j_E|)^{-1}m_\Sigma$ if $|j_E|>1$.}
\end{cases}
\end{equation*}

Let $Z=\{P_1,\dots,P_N\}$ be a set of $N$ distinct points in $E(\KK)$.  The {\em local discrepancy} of the set $Z$ is defined by
\begin{equation}\label{LocalDiscDef}
D(Z) = \frac{1}{N^2}\sum_{\stackrel{1\leq m,n\leq N}{m\neq n}}\lambda(P_m-P_n) + \frac{1}{12 N}\log^+|j_E|.
\end{equation}
By the decomposition $(\ref{ChinburgRumely})$ and the fact that $j(x,x)=\frac{1}{12}\log^+|j_E|$ for all $x\in\Esf$ we also have
\begin{equation}\label{discdecomp}
D(Z) =  \frac{1}{N^2}\sum_{\stackrel{1\leq m,n\leq N}{m\neq n}}i(P_m,P_n) +   \frac{1}{N^2}\sum_{1\leq m,n\leq N}j(P_m,P_n).
\end{equation}
We shall see that $D(Z)$ is nonnegative, and that it gives a quantitative measure of the $\mu$-equidistribution of the set $Z$.  This notion of local discrepancy was introduced in Baker-Petsche \cite{BakerPetsche} to study equidistribution on elliptic curves over number fields.

\subsection{Fourier analysis on $\RR/\ell\ZZ$}\label{EquidistRZ}

For the proof of our main equidistribution result, and again for one of our global applications in $\S$\ref{SIntSection}, we will use Fourier analysis on the circle group $\RR/\ell\ZZ$, where $\ell>0$.  In this section we will recall the basic theory and prove a lemma.

Given a Lebesgue integrable function $F:\RR/\ell\ZZ\to\RR$, its $k$-th Fourier coefficient is defined by $\hat{F}(k)=\frac{1}{\ell}\int_0^\ell F(x)e^{-2\pi ikx/\ell}dx$, and its Fourier series is given by $F(x) \sim \sum_{k\in\ZZ}\hat{F}(k)e^{2\pi ikx/\ell}$.  If $\sum_{k\in\ZZ}|\hat{F}(k)|<+\infty$ then the right-hand-side converges absolutely and the symbol ``$\sim$'' can be interpreted as an equality.  Assuming $F$ is square-integrable on $\RR/\ell\ZZ$ we have Parseval's formula
\begin{equation*}
\frac{1}{\ell}\int_0^\ell|F(x)|^2dx = \sum_{k\in\ZZ}|\hat{F}(k)|^2.
\end{equation*}
Given two functions $F,G:\RR/\ell\ZZ\to\RR$, recall that their convolution $F*G:\RR/\ell\ZZ\to\RR$ is defined by $F*G(x)=\frac{1}{\ell}\int_0^\ell F(y)G(x-y)dy$ and has Fourier coefficients $\widehat{F*G}(k)=\hat{F}(k)\hat{G}(k)$.

Let $\Psi(x) = x-[x]-1/2$ and $\Phi(x)=(x-[x])^2-(x-[x])+1/6$ be the first and second $\ZZ$-periodic Bernoulli polynomials respectively, and let $\Psi_{\ell}(x)=\Psi(x/\ell)$ and $\Phi_\ell(x)=\Phi(x/\ell)$ denote their $\ell$-periodic analogues.  Then $\hat{\Psi}_\ell(0)=\hat{\Phi}_\ell(0)=0$, while $\hat{\Psi}_\ell(k)=\frac{-1}{2\pi ik}$ and $\hat{\Phi}_\ell(k)=\frac{1}{2\pi^2 k^2}$ for $k\neq0$; thus $\Phi$ has an absolutely convergent Fourier series but $\Psi$ does not.

\begin{lem}\label{GSigmaLemma}
Suppose that $F:\RR/\ell\ZZ\to\RR$ is continuous and differentiable almost everywhere, and that $F'$ is Lebesgue integrable.  Given $N$ (possibly indistinct) points $p_1,\dots,p_N\in\RR/\ell\ZZ$, define $G:\RR/\ell\ZZ\to\RR$ by $G(x)=-\frac{1}{N}\sum_{n=1}^{N}\Psi_\ell(p_n-x)$.  Then
\begin{equation}\label{G1}
\frac{1}{N}\sum_{n=1}^{N}F(p_n)-\frac{1}{\ell}\int_0^\ell F(x)dx = \int_0^\ell F'(x)G(x)dx
\end{equation}
and
\begin{equation}\label{G2}
\int_0^\ell|G(x)|^2 dx = \frac{\ell}{2N^2}\sum_{1\leq i,j\leq N}\Phi_\ell(p_i-p_j).
\end{equation}
\end{lem}
\begin{proof}
We have $F(x)-\hat{F}(0)=-\ell F'*\Psi_\ell(x)$ for all $x\in \RR/\ell\ZZ$; this follows from integration by parts, or alternatively by noting that $\hat{F'}(k)=(2\pi i k/\ell)\hat{F}(k)$ and comparing the Fourier coefficients of both sides.  Thus
\begin{equation*}
\begin{split}
\frac{1}{N}\sum_{n=1}^{N}\{F(p_n)-\hat{F}(0)\} & = -\frac{\ell}{N}\sum_{n=1}^{N}F'*\Psi_\ell(p_n) \\
	& = \int_0^\ell F'(x)\bigg\{\frac{-1}{N}\sum_{n=1}^{N}\Psi_\ell(p_n-x) \bigg\}dx, \\
	& = \int_0^\ell F'(x)G(x)dx,
\end{split}
\end{equation*}
which implies $(\ref{G1})$.  Note that $\hat{G}(k) = -\frac{1}{2\pi ik}\{\frac{1}{N}\sum_{n=1}^{N}e^{-2\pi ik p_n/\ell}\}$ for $k\neq0$ and $\hat{G}(0)=0$, and thus by Parseval's formula we have
\begin{equation*}
\begin{split}
\frac{1}{\ell}\int_0^\ell|G(x)|^2 dx & = \sum_{k\neq0}\frac{1}{4\pi^2k^2}\Big|\frac{1}{N}\sum_{n=1}^{N}e^{-2\pi ik p_n/\ell}\Big|^2 \\
	& = \frac{1}{2N^2}\sum_{1\leq i,j\leq N}\sum_{k\neq0}\frac{1}{2\pi^2k^2}e^{2\pi ik (p_i-p_j)/\ell} \\
	& = \frac{1}{2N^2}\sum_{1\leq i,j\leq N}\Phi_\ell(p_i-p_j),
\end{split}
\end{equation*}
which is $(\ref{G2})$.
\end{proof}

\subsection{The main local inequality}  

We are now ready to state the main local results of this paper.  First, given a connected metrized subgraph $\Gamma$ of $\Esf\setminus E(\KK)$ containing $\Sigma$, written as in $(\ref{GammaUnion})$, define
\begin{equation}\label{ell0Def}
\ell_0(\Gamma) = \max_{b\in S}\rho(r_\Sigma(b),b).
\end{equation}
In other words, $\ell_0(\Gamma)$ is the greatest distance of any point $x\in\Gamma$ to the skeleton $\Sigma$; thus $\ell_0(\Gamma)=0$ if and only if $\Gamma=\Sigma$.

\begin{thm}\label{MainLocalThm}
Let $F\in\Scal_\Gamma(\Esf,\RR)$ be a test function associated to a connected metrized subgraph $\Gamma$ of $\Esf\setminus E(\KK)$ containing $\Sigma$.  Let $Z$ be a nonempty finite subset of $E(\KK)$.  Then
\begin{equation}\label{MainLocalThmIneq}
\bigg|\frac{1}{|Z|}\sum_{P\in Z}F(P)-\int F d\mu\bigg| \leq \|F'\|_{L^2(\Gamma)} \bigg(D(Z) + \frac{\ell_0(\Gamma)}{|Z|}\bigg)^{1/2}.
\end{equation}
\end{thm}

\begin{proof}

First, if $\Gamma$ consists of a single point then $F$ is constant, since $F=F\circ r_\Gamma$ factors through $\Gamma$.  In this case the left-hand-side of $(\ref{MainLocalThmIneq})$ is zero and, as the right-hand-side is nonnegative, the theorem is trivial.  So we may now assume that $\ell(\Gamma)>0$, or equivalently that either $E/\KK$ has multiplicative reduction or that $\Gamma\setminus\Sigma$ is nonempty.  Let $Z=\{P_1,\dots,P_N\}$.  We are going to define a piecewise continuous function $G:\Gamma\to\RR$, depending on the graph $\Gamma$ and the set $Z$ but not otherwise on the function $F$, which satisfies the identity
\begin{equation}\label{FGIdentity}
\frac{1}{N}\sum_{n=1}^{N}F(P_n)-\int F d\mu = \int_\Gamma F'(x)G(x) dm_\Gamma(x)
\end{equation}
and the inequality
\begin{equation}\label{GInequality} 
\|G\|_{L^2(\Gamma)}^2 \leq D(Z) + \frac{\ell_0(\Gamma)}{N}.
\end{equation}
The inequality $(\ref{MainLocalThmIneq})$ will then follow immediately from Cauchy's inequality.

We first define $G$ on the skeleton $\Sigma$.  This is only necessary in the multiplicative reduction case, since in the good reduction case $\Sigma$ has $m_\Gamma$-measure zero.  Supposing that $E/\KK$ has multiplicative reduction, we may identify $\Sigma\simeq\RR/\ell\ZZ$ as described in $\S$\ref{BerkSectionMR}, where $\ell=\log|j_E|>0$.  Thus $\mu=(1/\ell)m_\Sigma$ is the normalized uniform measure on $\Sigma$.  For each $1\leq n\leq N$ let $p_n=r_\Sigma(P_n)\in\Sigma$, and define $G:\Sigma\to\RR$ as in Lemma~\ref{GSigmaLemma}.  Then
\begin{equation}\label{FGIdentitySigma}
\frac{1}{N}\sum_{n=1}^{N}F(r_\Sigma(P_n))-\int F d\mu = \int_\Sigma F'(x)G(x)dm_\Gamma(x)
\end{equation}
by $(\ref{G1})$, and since $j(P,Q)$ factors through the retraction map $r_\Sigma:\Esf\to\Sigma$, we have
\begin{equation}\label{GInequalitySigma}
\begin{split}
\int_\Sigma|G(x)|^2 dm_\Gamma(x) & = \frac{1}{N^2}\sum_{1\leq m,n\leq N}j(r_\Sigma(P_m),r_\Sigma(P_n)) \\
	& = \frac{1}{N^2}\sum_{1\leq m,n\leq N}j(P_m,P_n)
\end{split}
\end{equation}
by $(\ref{G2})$.  If $\Gamma=\Sigma$ then the proof of the theorem is complete, since $(\ref{FGIdentitySigma})$ implies $(\ref{FGIdentity})$, and $(\ref{GInequality})$ follows from $(\ref{GInequalitySigma})$ and the fact that the right-hand-side of $(\ref{GInequalitySigma})$ is at most $D(Z)$ by $(\ref{discdecomp})$.  Note that $(\ref{GInequalitySigma})$ and the decomposition $(\ref{discdecomp})$, along with the nonnegativity of $i(x,y)$, shows that local discrepancy $D(Z)$ is nonnegative.

We will now define $G$ on $\Gamma\setminus\Sigma$, which we may assume is nonempty since the $\Gamma=\Sigma$ case has been settled.  Define $G:\Gamma\setminus\Sigma\to\RR$ by 
\begin{equation*}
G(x)  = \frac{1}{N}\sum_{n=1}^{N}\chi_{[r_\Sigma(P_n), r_\Gamma(P_n)]}(x).
\end{equation*}
Thus $NG(x)$ counts the number of points $P_n$ in $Z$ such that $x$ lies in the subsegment of $\Gamma$ from $r_\Sigma(P_n)$ to $r_\Gamma(P_n)$.  Since $r_\Gamma(P_n)=r_\Sigma(P_n)$ if $r_\Gamma(P_n)\in\Sigma$, we have
\begin{equation}\label{FGIdentityGamma}
\frac{1}{N}\sum_{n=1}^{N}F(r_\Gamma(P_n))-\frac{1}{N}\sum_{n=1}^{N}F(r_\Sigma(P_n)) = \int_{\Gamma\setminus\Sigma} F'(x)G(x)dm_\Gamma(x)
\end{equation}
by elementary calculus.  

For each $n$ let us abbreviate $\sigma_n=r_\Sigma(P_n)\in\Sigma$ and $\gamma_n=r_\Gamma(P_n)\in\Gamma$.  Given two (possibly equal) points $P_m,P_n$ in $Z$, if they are elements of two distinct connected components of $\Esf\setminus\Sigma$, then the intersection $[\sigma_m, \gamma_m]\cap[\sigma_n, \gamma_n]$ is either the empty set or it contains only the single point $\sigma=\sigma_m=\sigma_n$.  In either case the intersection has $m_\Gamma$-measure zero, and $i(P_m,P_n)=0$.  On the other hand suppose that $P_m,P_n$ are elements of the same connected component $\Bsf^\circ$ of $\Esf\setminus\Sigma$.  Then $\sigma_m=\sigma_n$ ($=\sigma$) and 
\begin{equation*}
\begin{split}
[\sigma_m, \gamma_m]\cap[\sigma_n, \gamma_n] & = [\sigma,\gamma_m\vee\gamma_n] \\
	& = [\sigma,r_\Gamma(P_m\vee P_n)] \hskip1cm\text{ by $(\ref{RetVeeRule})$} \\
	& \subseteq  [\sigma,P_m\vee P_n].
\end{split}
\end{equation*}
Recall that $\rho(\sigma,P_m\vee P_n)=i(P_m,P_n)$ by definition for $m\neq n$.  When $m=n$ we use the inequality $\rho(\sigma,r_\Gamma(P_m\vee P_n))\leq\ell_0(\Gamma)$ by the definition $(\ref{ell0Def})$ of $\ell_0(\Gamma)$.  Summing over all pairs we have
\begin{equation}\label{GInequalityGamma}
\begin{split}
\int_{\Gamma\setminus\Sigma}|G(x)|^2dm_\Gamma(x) & = \frac{1}{N^2}\sum_{1\leq m,n\leq N}\int_{\Gamma\setminus\Sigma}\chi_{[\sigma_m, \gamma_m]}(x)\chi_{[\sigma_n, \gamma_n]}(x)dm_\Gamma(x) \\
	& = \frac{1}{N^2}\sum_{1\leq m,n\leq N}\int_{\Gamma\setminus\Sigma}\chi_{[\sigma_m, \gamma_m]\cap[\sigma_n, \gamma_n]}(x)dm_\Gamma(x) \\
	& \leq  \frac{1}{N^2}\sum_{m\neq n}i(P_m, P_n) + \frac{\ell_0(\Gamma)}{N}.
\end{split}
\end{equation}
The identities $(\ref{FGIdentitySigma})$ and $(\ref{FGIdentityGamma})$ imply $(\ref{FGIdentity})$, the identities $(\ref{discdecomp})$ and $(\ref{GInequalitySigma})$ along with the inequality $(\ref{GInequalityGamma})$ imply $(\ref{GInequality})$, and Cauchy's inequality completes the proof of $(\ref{MainLocalThmIneq})$. 
\end{proof}

%%%%%%%%%%%%%%%%%%%%%%%%%%%%%%%%%%%%%%%%%%%%%%%%%%%%%%%%%%%%
%%%%%%%%%%%%%%%%%%%%%%%%%%%%%%%%%%%%%%%%%%%%%%%%%%%%%%%%%%%%
%%%%%%%%%%%%%%%%%%%%%%%%%%%%%%%%%%%%%%%%%%%%%%%%%%%%%%%%%%%%

\section{The Global Equidistribution Theorem}\label{GlobalSection}

\subsection{Global function fields}

Let $K$ be a field equipped with a set $M_K$ of inequivalent, nontrivial, discrete valuations on $K$ such that:
\begin{itemize}
	\item $\{v\in M_K \mid v(a)\neq0\}$ is finite for each $a\in K^\times$;
	\item $\sum_{v\in M_K}v(a)=0$ for each $a\in K^\times$ (product formula);
	\item $[L:K]=\sum_{w|v}[L_w:K_v]$ for each finite extension $L/K$ and each $v\in M_K$ (local degree formula).
\end{itemize}
The sum in the local degree formula is over all valuations $w$ on $L$ whose restriction to $K$ coincides with $v\in M_K$ (in this case we write $w|v$), and $K_v$ (resp. $L_w$) denotes the completion of $K$ at $v$ (resp. $L$ at $w$).  We will refer to the elements of $M_K$ as {\em places} of $K$, and instead of the additive valuation $v$, we will usually use the corresponding non-archimedean absolute value $|\cdot|_v=e^{-v(\cdot)}$.  We will call a field $K$ satisfying these axioms a {\em global function field}.  The most commonly encountered example is given by the field $K=k(C)$ of rational functions on an integral, proper, geometrically connected curve $C$ defined over an arbitrary field $k$; in this case the places of $K$ correspond to the scheme-theoretic closed points of $C$.  There are other examples, however; see \cite{BombieriGubler} $\S$~1.4.6 and \cite{Lang} $\S$~2.3.

The local degree formula always holds for separable finite extensions $L/K$ (\cite{BombieriGubler} Cor. 1.3.2).  It was pointed out to us by the referee that, even when $L/K$ is inseparable and the local degree formula fails, a suitable product formula can be recovered for $L$, provided one takes care to renormalize the valuations $w\in M_L$ appropriately.  This is worked out in \cite{BombieriGubler} $\S$~1.4 and \cite{Serre} $\S$~2.1.

\medskip

We will now summarize the basic properties of global function fields; for detailed proofs see Lang \cite{Lang}.  First, it follows from the local degree formula that the product formula holds in the form
\begin{equation*}
\prod_{w\in M_L}|a|_w^{[L_w:K_w]}=1 
\end{equation*}
for each finite extension $L/K$, where $M_L$ denotes the set of absolute values on $L$ restricting to an absolute value in $M_K$.

Given a finite extension $L/K$, define the {\em constant field} of $L$ by $L_0=\{a\in L \mid |a|_w\leq1 \text{ for all }w\in M_L \}$.  It follows from the ultrametric inequality and the product formula that $L_0$ is a subfield of $L$, and that $|a|_w=1$ for all $a\in L_0^\times$ and all $w\in M_L$.  Let $\Kbar$ be the algebraic closure of $K$, and let $\Kbar_0=\cup_{L/K}L_0$ denote the constant field of $\Kbar$.  

For each $v\in M_K$, let $\KK_v$ denote the completion of the algebraic closure $\Kbar_v$ of $K_v$; the absolute value $|\cdot|_v$ extends uniquely to the field $\KK_v$, which is complete and algebraically closed (\cite{BGR}, $\S$3.4).  Given a finite extension $L/K$ and a place $v\in M_K$, each $K$-embedding $\epsilon:L\hookrightarrow\KK_v$ determines a valuation $w\in M_L$ over $v$ by $|\cdot|_w=|\epsilon(\cdot)|_v$, and conversely every place $w\in M_L$ arises in this way.  If $\epsilon,\delta:L\hookrightarrow\KK_v$ are two such $K$-embeddings, they determine the same element $w\in M_L$ if and only if $\epsilon=\sigma\circ\delta$ for some $\sigma$ in the group $\Aut(\KK_v/K_v)$ of continuous $K_v$-automorphisms of $\KK_v$.

Finally, suppose that $L/K$ is a finite normal extension, and suppose that $\epsilon,\delta:L\hookrightarrow\KK_v$ are two $K$-embeddings.  Then $\epsilon(L)=\delta(L)$, and therefore $\epsilon=\delta\circ\alpha$ for some $\alpha\in\Aut(L/K)$.

\subsection{The absolute Weil height}

The absolute Weil height $h:\PP^1(\Kbar)\to\RR$ is defined as follows.  Given a point $x=(x_0:x_1)\in\PP^1(\Kbar)$, select a finite extension $L/K$ containing the coordinates of $x$, and set 
\begin{equation}\label{WeilHeight}
h(x)=\sum_{w\in M_L}\frac{[L_w:K_w]}{[L:K]}\log\max\{|x_0|_w,|x_1|_w\}.
\end{equation}
By the product formula and the local degree formula the value of $h(x)$ does not depend on the choice of coordinates or the choice of $L$ containing them.  We also denote by $h:\Kbar\to\RR$ the affine height given by $h(x)=h(1:x)$ for $x\in\Kbar$.  It is easy to check that $h(x)\geq0$ for all $x\in\Kbar$, with $h(x)=0$ if and only if $x\in \Kbar_0$.

\subsection{Elliptic curves and the N\'eron-Tate height}

Let $E/K$ be an elliptic curve over $K$.  The N\'eron-Tate height function $\hhat:E(\Kbar)\to\RR$ is defined by
\begin{equation*}
\hhat(P) = \frac{1}{2}\lim_{n\to+\infty}\frac{h(x(nP))}{n^2},
\end{equation*}
where $x:E(\Kbar)\to\PP^1(\Kbar)$ is the $x$-coordinate function associated to a Weierstrass equation for $E/K$.  For the proofs of the following basic facts about $\hhat$, see Silverman \cite{SilvermanI} $\S$~VIII.9 (strictly speaking \cite{SilvermanI} treats the number field case only, although the arguments work in the function field setting).  For a general treatment of N\'eron-Tate heights on abelian varieties over arbitrary global fields, see Bombieri-Gubler \cite{BombieriGubler} $\S$~9.2.

The N\'eron-Tate height function satisfies $\hhat(P)=(1/2)h(x(P))+O(1)$ and $\hhat(nP)=n^2\hhat(P)$ for all $P\in E(\Kbar)$, and in fact $\hhat$ can be characterized as the unique function satisfying these two properties.  The definition of $\hhat$ is independent of the $x$-coordinate function used to define it.  

Using the above properties of $\hhat$ it is easy to see that $\hhat(P)\geq0$ for all $P\in E(\Kbar)$, and
\begin{equation}\label{TorsionHt}
P\in E(\Kbar)_\tor \hskip1cm \Rightarrow \hskip1cm \hhat(P)=0.
\end{equation}
However, unlike in the number field case, the converse of $(\ref{TorsionHt})$ may fail.  For example, if $E$ is defined over the constant field $K_0$ of $K$, then every point in $E(\Kbar_0)$ has canonical height zero, and $E(\Kbar_0)$ may contain nontorsion points (if $K_0$ is infinite).  See Lang \cite{Lang} $\S$~6.5 and Baker \cite{Baker2} App. B for results which imply the converse to $(\ref{TorsionHt})$ under certain conditions on $E/K$.

Finally, we note for future use a fundamental property of the N\'eron-Tate height, the parallelogram law:
\begin{equation}\label{ParLaw}
\hhat(P+Q) + \hhat(P-Q) = 2\hhat(P)+2\hhat(Q) \hskip1cm (P,Q\in E(\Kbar)).
\end{equation}

\subsection{Local height functions}

For each place $v\in M_K$ we may consider $E$ as an elliptic curve over $\KK_v$, and we generally affix a subscript $v$ to each local object associated to $E/\KK_v$ introduced in $\S$\ref{BerkAnalyticSection} and $\S$\ref{MainLocalSection}.  Thus $\Esf_v$ denotes the Berkovich analytic space associated to $E/\KK_v$, $\lambda_v:E(\KK_v)\setminus\{O\}\to\RR$ is the N\'eron local height function discussed in $\S$\ref{NeronFunction}, and likewise for $i_v$, $j_v$, $\Sigma_v$, $\mu_v$, etc.

Let $L/K$ be a finite extension.  For each $w\in M_L$, let $\lambda_w:E(L_w)\setminus\{O\}\to\RR$ denote the associated N\'eron local height function.  By the uniqueness of such functions, if the place $w$ corresponds to the $K$-embedding $\epsilon:L\hookrightarrow\KK_v$, then $\lambda_w(\cdot)=\lambda_v(\epsilon(\cdot))$.

Given a global point $P\in E(L)\setminus\{O\}$, we have a decomposition 
\begin{equation}\label{LocalGlobalHeight}
\hhat(P) = \sum_{w\in M_L}\frac{[L_w:K_w]}{[L:K]}\lambda_w(P)
\end{equation}
of the global height into local terms.  For the proof of $(\ref{LocalGlobalHeight})$ see \cite{SilvermanII} $\S$VI.2 (the proof is written for number fields but works in this setting) or \cite{CallSilverman}.

\subsection{The global equidistribution theorem}

In this section we will prove the quantitative global equidistribution result (Theorem~\ref{QuanSUZIntro}), along with its qualitative corollary (Corollary~\ref{MainGlobalCorIntro}).  Given a nonempty finite subset $Z$ of $E(\Kbar)$, recall that its height $\hhat(Z)$ is defined to be the average $\hhat(Z)=(1/|Z|)\sum_{P\in Z}\hhat(P)$ of the height of its points.  The following is a more precise statement of Theorem~\ref{QuanSUZIntro}.  

\begin{thm}\label{QuanSUZ}
Let $E/K$ be an elliptic curve over a global function field, let $v\in M_K$ be a place of $K$, and let $\epsilon:\Kbar\hookrightarrow\KK_v$ be a $K$-embedding.  Let $Z$ be a nonempty finite $\Aut(\Kbar/K)$-stable subset of $E(\Kbar)$, and let $F\in \Scal(\Esf_v,\RR)$ be a test function on $\Esf_v$ with associated connected metrized subgraph $\Gamma$ containing $\Sigma_v$.  Then
\begin{equation*}
\bigg|\frac{1}{|Z|}\sum_{P\in Z}F(\epsilon(P))-\int F d\mu_v\bigg| \leq \|F'\|_{L^2(\Gamma)} \bigg(4\hhat(Z) + \frac{h(j_E)}{12|Z|} + \frac{\ell_0(\Gamma)}{|Z|}\bigg)^{1/2}.
\end{equation*}
\end{thm}

\begin{proof}
Let $L/K$ be a finite normal extension such that $Z=\{P_1,\dots,P_N\}\subseteq E(L)$.  Given a place $w\in M_L$ corresponding to some $K$-embedding $\delta:L\hookrightarrow\KK_u$, where $u\in M_K$ is the place below $w$, define $D_w(Z)=D_u(\delta(Z))$, where $D_u$ denotes the local discrepancy defined on subsets of $E(\KK_u)$, as in $\S$\ref{CanMeasLocDisc}.

We are going to bound the sum
\begin{equation*}
D = \sum_{w\in M_L} \frac{[L_w:K_w]}{[L:K]}D_w(Z)
\end{equation*}
from above globally and from below locally.  First the global bound:
\begin{equation}\label{GlobalUB}
\begin{split}
D & = \sum_{w\in M_L}\frac{[L_w:K_w]}{[L:K]}\bigg\{\frac{1}{N^2}\sum_{m\neq n}\lambda_w(P_m-P_n) + \frac{\log^+|j_E|_w}{12 N}\bigg\} \\
	& = \frac{1}{N^2}\sum_{1\leq m,n\leq N}\hhat(P_m-P_n) + \frac{h(j_E)}{12N} \\
	& \leq \frac{1}{N^2}\sum_{1\leq m,n\leq N}(2\hhat(P_m)+2\hhat(P_n)) + \frac{h(j_E)}{12N} \\
	& = 4\hhat(Z) + \frac{h(j_E)}{12N}
\end{split}
\end{equation}
by the decompositions $(\ref{WeilHeight})$ and $(\ref{LocalGlobalHeight})$, the parallelogram law, and the nonnegativity of the height $\hhat$.  At places $w\in M_L$ with $w\nmid v$, we use the trivial lower bound $D_w(Z)\geq0$ stemming from the fact that the local discrepancy $(\ref{LocalDiscDef})$ is nonnegative.

Finally, we treat the places $w\in M_L$ with $w|v$.  We first observe that $D_w(Z)=D_v(\epsilon(Z))$ for all $w|v$, where $\epsilon:L\hookrightarrow\KK_v$ is the fixed $K$-embedding in the statement of the theorem.  [Proof: Any $w|v$ is induced by {\em some} $K$-embedding $\delta:L\hookrightarrow\KK_v$, and since $L/K$ is normal $\delta=\epsilon\circ\alpha$ for some $\alpha\in\Aut(L/K)$.  By hypothesis $Z$ is $\Aut(L/K)$-invariant.]  We deduce from the local degree formula that $\sum_{w|v} \frac{[L_w:K_v]}{[L:K]}D_w(Z)=D_v(\epsilon(Z))$, and assembling $(\ref{GlobalUB})$ with the local considerations we have $D_v(\epsilon(Z))\leq 4\hhat(Z) + h(j_E)/12N$.  The stated inequality in Theorem \ref{QuanSUZ} now follows from $(\ref{MainLocalThmIneq})$.
\end{proof}

\begin{proof}[Proof of Corollary~\ref{MainGlobalCorIntro}]
Let $F:\Esf_v\to\RR$ be a continuous function, and let $\theta>0$.  Since $\Scal(\Esf_v,\RR)$ is dense in $\Ccal(\Esf_v,\RR)$ (with respect to the supremum norm), there exists a function $F^*\in \Scal(\Esf_v,\RR)$ such that $|F(x)-F^*(x)|\leq\theta$ for all $x\in \Esf_v$.  Since $\hhat(Z_n)\to0$ and $|Z_n|\to+\infty$, it follows from Theorem~\ref{QuanSUZ} that the desired limit formula $(\ref{MainGlobalCorLimForm})$ holds with $F$ replaced by $F^*$.  Using the triangle inequality we have
\begin{eqnarray*}
\bigg|\frac{1}{|Z_n|}\sum_{P\in Z_n}F(\epsilon(P))-\int F d\mu_v\bigg| & \leq &  \bigg|\frac{1}{|Z_n|}\sum_{P\in Z_n}(F(\epsilon(P))-F^*(\epsilon(P)))\bigg| \\
	& & + \bigg|\frac{1}{|Z_n|}\sum_{P\in Z_n}F^*(\epsilon(P))-\int F^* d\mu_v\bigg| \\
	& & + \bigg|\int (F^*-F) d\mu_v\bigg|,
\end{eqnarray*}
and since $\int 1d\mu_v=1$ and $|F(x)-F^*(x)|\leq\theta$ we deduce that
\begin{equation}\label{MainGlobalCorIneq}
\limsup_{n\to+\infty}\bigg|\frac{1}{|Z_n|}\sum_{P\in Z_n}F(\epsilon(P))-\int F d\mu_v\bigg|\leq 2\theta.
\end{equation}
Since $\theta>0$ is arbitrary, the left-hand-side of $(\ref{MainGlobalCorIneq})$ is zero, completing the proof of $(\ref{MainGlobalCorLimForm})$.
\end{proof}

%%%%%%%%%%%%%%%%%%%%%%%%%%%%%%%%%%%%%%%%%%%%%%%%%%%%%%%%%%%%
%%%%%%%%%%%%%%%%%%%%%%%%%%%%%%%%%%%%%%%%%%%%%%%%%%%%%%%%%%%%
%%%%%%%%%%%%%%%%%%%%%%%%%%%%%%%%%%%%%%%%%%%%%%%%%%%%%%%%%%%%

\section{The Finiteness of $S$-integral Torsion Points}\label{SIntSection}

\subsection{Introduction}

In this section we will prove Theorem \ref{BakerIhRumelyFF2Intro}; for the convenience of the reader we will restate the theorem below.  Let $K$ be a global function field and let $E/K$ be an elliptic curve.  Given a place $v$ of $K$ and a point $R\in E(\KK_v)$, define 
\begin{equation}\label{Mconstant}
M_v(R) = \inf_{P\in E(\KK_v)_\tor}d_v(P,R),
\end{equation}
where $d_v$ is the $v$-adic metric on $E(\KK_v)$ defined in $\S$\ref{BerkSection}.  Thus $M_v(R)$ measures the distance between $R$ and its nearest torsion point in $E(\KK_v)$.  Plainly $M_v(R)=0$ if $R$ is a torsion point, but if $R$ is nontorsion, then the discreteness of $E(\KK_v)_\tor$ (Lemma \ref{TorsionDiscrete}) and the definition of the metric $d_v$ shows that $0<M_v(R)\leq1$.  If $Q\in E(\Kbar)$ is a global point, define
\begin{equation}\label{MconstantGlobal}
M_v(Q) = \inf_{\epsilon}M_v(\epsilon(Q)) \hskip1cm m_v(Q)=-\log M_v(Q),
\end{equation}
the infimum over all $K$-embeddings $\epsilon:\Kbar\hookrightarrow\KK_v$; plainly $0<M_v(Q)\leq1$ if $Q$ is nontorsion, since $Q$ has only finitely many embeddings into $E(\KK_v)$.  Thus $m_v(Q)$ is nonnegtive, and finite provided $Q$ is nontorsion.

\begin{thm}\label{BakerIhRumelyFF2}
Let $E/K$ be an elliptic curve over a global function field.  Let $(\ref{Weierstrass})$ be a Weierstrass equation for $E/K$, let $S$ be a finite set of places of $K$ such that $(\ref{Weierstrass})$ is defined over $\Ocal_S$, and such that $|\Delta|_v=1$ for all $v\not\in S$, where $\Delta\in K$ is the discriminant of $(\ref{Weierstrass})$.  Let $\Ecal/\Ocal_S$ denote the associated $S$-integral Weierstrass model for $E$.  If $Q$ is a point in $E(\Kbar)$ with $\hhat(Q)>0$, then 
\begin{equation}\label{BakerIhRumelyFF2Ineq}
|\Ecal(\bar{\Ocal}_S,Q)_\tor| \leq \frac{1}{\hhat(Q)^{2}}\bigg(\frac{|S|h(j_E)}{12}+\sum_{v\in S} m_v(Q)\bigg)^2.
\end{equation}
\end{thm}

Our proof follows the basic strategy used by Baker-Ih-Rumely \cite{BakerIhRumely} in the number field case, with some slight differences.  The idea is to use the equidistribution theorem at each place $v\in S$, testing a large set of global torsion points against truncated N\'eron local height functions.  

\subsection{Truncated N\'eron local height functions}

Fix a place $v$ of $K$.  In order to prove Theorem \ref{BakerIhRumelyFF2} we will use certain test functions in the space $\Scal(\Esf_v,\RR)$ which are formed by truncating linear combinations of N\'eron local height functions.  The work in this section is entirely local and thus we suppress the dependence on the place $v$ in all of our notation.

Let $\Rcal$ be a nonempty finite subset of $E(\KK)$, and define $F:\Esf\setminus\Rcal\to\RR$ by 
\begin{equation*}
F(x)=\frac{1}{|\Rcal|}\sum_{R\in \Rcal}\lambda(x-R).
\end{equation*}
Note that $F$ is continuous on $\Esf\setminus\Rcal$ but has logarithmic singularities at the points of $\Rcal$.  However, fix a constant $m>0$ and define the truncated function $G:\Esf\to\RR$ by
\begin{equation*}
G(x)=\frac{1}{|\Rcal|}\sum_{R\in \Rcal}\{j(x,R) + i^*(x,R)\},
\end{equation*}
where $i^*(x,R)=\min\{i(x,R),m\}$.  For each $R\in\Rcal$ let $\kappa_R$ be the unique point in the infinitely long line segment $[r_\Sigma(R),R]$ in $\Esf$ which satisfies $\rho(r_\Sigma(R),\kappa_R)=m$.  Let 
\begin{equation*}
\Gamma=\Sigma\cup\big(\bigcup_{R\in\Rcal}[r_\Sigma(R),\kappa_R]\big)
\end{equation*}
be the smallest connected metrized subgraph of $\Esf\setminus E(\KK)$ which contains $\Sigma$ and all of the $\kappa_R$.  

\begin{lem}\label{LocalTruncLemma}
The function $G:\Esf\to\RR$ is in the space $\Scal_\Gamma(\Esf,\RR)$ of test functions associated to $\Gamma$, $\ell_0(\Gamma)=m$, and
\begin{equation}\label{LocalTruncLemmaIneq}
\|G'\|_{L^2(\Gamma)}^2 \leq \frac{1}{12}\log^+|j_E| + m.
\end{equation}
\end{lem}
\begin{proof}
In order to show that $G\in\Scal_\Gamma(\Esf,\RR)$ we must show that $G=G\circ r_\Gamma$, where $r_\Gamma:\Esf\to\Gamma$ is the retraction map.  Fix a point $x\in\Esf$; we are going to show that for each $R\in\Rcal$,
\begin{equation}\label{ImIndentity}
i^*(x,R) = i^*(r_\Gamma(x),R);
\end{equation}
in view of that fact that the function $j(x,y)$ factors through the retraction map $r_\Sigma:\Esf\to\Sigma$, and therefore the map $r_\Gamma:\Esf\to\Gamma$ as well, this will complete the proof that $G=G\circ r_\Gamma$.

First, suppose that $x$ and $R$ are elements of the same connected component $\Bsf^\circ$ of $\Esf\setminus\Sigma$ as in the decomposition $(\ref{ConnCompPar1})$ or $(\ref{ConnCompPar2})$.  Then set $x_R=x\vee R$, the least upper bound of $x$ and $R$ with respect to the partial order on $\Bsf^\circ$.  If $x_R\in [r_\Sigma(R),\kappa_R]$, then $i(x,R)=\rho(r_\Sigma(x_R),x_R)\leq \rho(r_\Sigma(x_R),\kappa_R)=m$, and since $x_R=x\vee R=r_\Gamma(x)\vee R$, we have $i(x,R)=i(r_\Gamma(x),R)$.  Since both of these quantities are at most $m$, we deduce $(\ref{ImIndentity})$ in this case.  On the other hand, suppose that $x_R\not\in [r_\Sigma(R),\kappa_R]$.  Then by a similar argument as above, both $i(x,R)$ and $i(r_\Gamma(x),R)$ are greater than $m$, whereby both sides of $(\ref{ImIndentity})$ are equal to $m$.  Finally, if $x$ and $R$ are {\em not} both elements of the same connected component $\Bsf^\circ$ of $\Esf\setminus\Sigma$, then both sides of $(\ref{ImIndentity})$ vanish; the proof that $G=G\circ r_\Gamma$ is complete.  The other axioms satisfied by functions in $\Scal_\Gamma(\Esf,\RR)$ are trivially verified for $G$.  The equality $\ell_0(\Gamma)=m$ follows at once from the fact that $\rho(r_\Sigma(R),\kappa_R)=m$ for all $R\in\Rcal$.

Finally, we are going to show that 
\begin{equation}\label{GPrimeIneqs}
\|G'\|_{L^2(\Sigma)}^2 \leq \frac{1}{12}\log^+|j_E| \hskip5mm \text{ and } \hskip5mm \|G'\|_{L^2(\Gamma\setminus\Sigma)}^2 \leq m,
\end{equation}
which together imply $(\ref{LocalTruncLemmaIneq})$.  First, both sides of the first inequality in $(\ref{GPrimeIneqs})$ vanish in the good reduction case, so we may assume that $E/\KK$ has multiplicative reduction; thus $\ell:=\log|j_E|>0$.  Since each $i(x,R)$ ($R\in\Rcal$) vanishes on the skeleton, for $x\in\Sigma\simeq\RR/\ell\ZZ$ we have
\begin{equation*}
G(x) = \frac{1}{|\Rcal|}\sum_{R\in \Rcal}j(x,R) = \frac{1}{|\Rcal|}\sum_{R\in \Rcal}\frac{\ell}{2}\Phi\bigg(\frac{x - r_\Sigma(R)}{\ell}\bigg),
\end{equation*}
and since $\Phi'(x)=2\Psi(x)$, we deduce that $G'(x) = \frac{1}{|\Rcal|}\sum_{R\in \Rcal}\Psi_\ell(x - r_\Sigma(R))$; here we are using the notation of $\S$\ref{EquidistRZ}.  Therefore $\hat{G}'(0)=0$ and for $k\neq0$,
\begin{equation*}
\hat{G}'(k) = \frac{-1}{2\pi ik}\bigg\{\frac{1}{|\Rcal|}\sum_{R\in \Rcal}e^{-2\pi ikr_\Sigma(R)/\ell}\bigg\};
\end{equation*}
in particular $|\hat{G}'(k)|\leq1/2\pi|k|$ for $k\neq0$.  It follows from Parseval's formula that 
\begin{equation*}
\int_0^\ell|G'(x)|^2dx = \ell\sum_{k\in\ZZ}|\hat{G}'(k)|^2 \leq \ell\sum_{k\neq0}\frac{1}{4\pi^2k^2}=\frac{\ell}{12},
\end{equation*}
which proves the first inequality in $(\ref{GPrimeIneqs})$ in this case.  

To prove the second inequality in $(\ref{GPrimeIneqs})$, fix $R\in\Rcal$, and identify the line segment $[r_\Sigma(R),\kappa_R]$, which by the definition of $\kappa_R$ has length $m$, with the real interval $[0,m]$.  Defining $g_R:[r_\Sigma(R),\kappa_R]\to\RR$ by $g_R(x)=i(x,R)$, we have $g_R(x)=\rho(r_\Sigma(R),x)=x$, and thus $g_R'(x)\equiv1$ on $[r_\Sigma(R),\kappa_R]$.  Therefore $\int_{[r_\Sigma(R),\kappa_R]}|g_R'(x)|^2dm_\Gamma(x)=m$.  Extending the definition of $g_R$ by declaring that $g_R(x)=0$ for all $x\not\in[r_\Sigma(R),\kappa_R]$, we have 
\begin{equation*}
G'(x) = \frac{1}{|\Rcal|}\sum_{R\in \Rcal}g_R'(x)
\end{equation*}
for $x\in\Gamma\setminus\Sigma$, and thus
\begin{equation*}
\begin{split}
\int_{\Gamma\setminus\Sigma}|G'(x)|^2dm_\Gamma(x) & = \int_{\Gamma\setminus\Sigma}\Big|\frac{1}{|\Rcal|}\sum_{R\in \Rcal}g_R'(x)\Big|^2dm_\Gamma(x) \\
	& = \frac{1}{|\Rcal|^2}\sum_{Q,R\in \Rcal}\int_{\Gamma\setminus\Sigma}g_Q'(x)g_R'(x)dm_\Gamma(x) \\
	& \leq \frac{1}{|\Rcal|^2}\sum_{Q,R\in \Rcal}m =m,
\end{split}
\end{equation*}
which is the second inequality in $(\ref{GPrimeIneqs})$.
\end{proof}

\subsection{The proofs of Theorems \ref{BakerIhRumelyFF} and \ref{BakerIhRumelyFF2}}

\begin{proof}[Proof of Theorem \ref{BakerIhRumelyFF}]
This follows immediately from Theorem \ref{BakerIhRumelyFF2} since enlarging the set $S$ only makes the set $\Ecal(\bar{\Ocal}_S,Q)_\tor$ larger.
\end{proof}

\begin{proof}[Proof of Theorem \ref{BakerIhRumelyFF2}]
Note that the set $\Ecal(\bar{\Ocal}_S,Q)_\tor$ is $\Aut(\Kbar/K)$-invariant.  Given a finite $\Aut(\Kbar/K)$-invariant subset $Z$ of $\Ecal(\bar{\Ocal}_S,Q)_\tor$, we are going to show that $|Z|$ is at most the right-hand-side of $(\ref{BakerIhRumelyFF2Ineq})$, which will prove the theorem.  We may assume that $Z$ is nonempty, for if $\Ecal(\bar{\Ocal}_S,Q)_\tor$ is empty the theorem is trivial.

Let $\Qcal$ be the set of $\Aut(\Kbar/K)$-conjugates of $Q$, and let $L/K$ be a finite normal extension such that $\Qcal\cup Z\subseteq E(L)$.  For each $w\in M_L$ define
\begin{equation*}
\Lambda_w(Z,\Qcal) = \frac{1}{|\Qcal||Z|}\sum_{P\in Z}\sum_{R\in \Qcal}\lambda_w(P-R),
\end{equation*}
and for each $v\in M_K$ let $\Lambda_v(Z,\Qcal)=\sum_{w|v}\frac{[L_w:K_v]}{[L:K]}\Lambda_w(Z,\Qcal)$.  Thus
\begin{equation}\label{HeightofQ}
\hhat(Q) = \frac{1}{|\Qcal||Z|}\sum_{P\in Z}\sum_{R\in \Qcal}\hhat(P-R) = \sum_{v\in M_K}\Lambda_v(Z,\Qcal)
\end{equation}
by $(\ref{LocalGlobalHeight})$; here we have used the fact that $\hhat(P-R)=\hhat(R)=\hhat(Q)$ for all $P\in Z$ and $R\in \Qcal$, since the points $P$ in $Z$ are torsion points and $\hhat$ is $\Aut(\Kbar/K)$-invariant.
  
We are going to use the $S$-integrality hypothesis to show that 
\begin{equation}\label{Claim1}
\Lambda_v(Z,\Qcal)=0 \hskip1cm\text{ for all }v\in M_K\setminus S,
\end{equation}
and we are going to use equidistribution (Theorem \ref{QuanSUZ}) to show that
\begin{equation}\label{Claim2}
\Lambda_v(Z,\Qcal)\leq \frac{1}{|Z|^{1/2}}\bigg(\frac{h(j_E)}{12}+m_v(Q)\bigg) \hskip1cm\text{ for all }v\in S.
\end{equation}
Assuming that these two claims hold for now, we deduce from $(\ref{HeightofQ})$ that
\begin{equation*}
\hhat(Q) \leq \sum_{v\in S}\frac{1}{|Z|^{1/2}}\bigg(\frac{h(j_E)}{12}+m_v(Q)\bigg) = \frac{1}{|Z|^{1/2}}\bigg(\frac{|S|h(j_E)}{12}+\sum_{v\in S}m_v(Q)\bigg),
\end{equation*}
and the desired inequality $(\ref{BakerIhRumelyFF2Ineq})$ follows immediately, completing the proof of the theorem.  Moreover, our assumption that $Z$ is nonempty ensures that the set $S$ is nonempty; for if $S$ were empty then $(\ref{HeightofQ})$ and $(\ref{Claim1})$ would contradict the assumption that $\hhat(Q)>0$.  In other words, this shows that if $S$ is empty then $\Ecal(\bar{\Ocal}_S,Q)_\tor$ is empty.

\medskip

We will now show $(\ref{Claim1})$.  Assuming $v\in M_K\setminus S$, the Weierstrass equation $(\ref{Weierstrass})$ is defined over $\Ocal_S$ with $|\Delta|_v=1$ and $|j_E|_v\leq1$.  Thus the model $\Ecal/\Ocal_S$ has good reduction at $v$, meaning the fiber $\bar{\Ecal}_v/k_v$ is an elliptic curve over the residue field $k_v=\OO_v/\MM_v$, and the associated reduction map $\pi_v:E(\KK_v)\to\bar{\Ecal}_v(k_v)$ is a group homomorphism.  Since the Zariski closure in $\Ecal$ of each point $P$ in $Z$ does not meet the Zariski closure of $Q$, this means that $\pi_v(\epsilon(P))\neq\pi_v(\epsilon(R))$ for every embedding $\epsilon:L\hookrightarrow\KK_v$ and every $R\in\Qcal$.  Thus if $w|v$ is a place of $L$ corresponding to the embedding $\epsilon$, then $\lambda_w(P-R)=\lambda_v(\epsilon(P)-\epsilon(R))=i_v(\epsilon(P),\epsilon(R))=0$ by $(\ref{ReductionIFunction})$ and Proposition~\ref{LocalHeightProp}.  We conclude that $\Lambda_w(Z,\Qcal)=0$ for all $w|v$, and thus $\Lambda_v(Z,\Qcal)=0$ for all $v\in M_K\setminus S$, which is $(\ref{Claim1})$.

\medskip

It now remains only to prove $(\ref{Claim2})$.  Assume that $v\in S$, and fix once and for all a $K$-embedding $\epsilon:\Kbar\hookrightarrow\KK_v$.  We claim that for each place $w|v$ of $L$, 
\begin{equation}\label{LambdaW}
\Lambda_v(Z,\Qcal)=\Lambda_w(Z,\Qcal)=\frac{1}{|\Qcal||Z|}\sum_{P\in Z}\sum_{R\in \Qcal}\lambda_v(\epsilon(P)-\epsilon(R)).
\end{equation}
To see this, note that each place $w|v$ of $L$ arises from {\em some} $K$-embedding $\delta:\Kbar\hookrightarrow\KK_v$ where $\lambda_w(\cdot)=\lambda_v(\delta(\cdot))$.  Since $L/K$ is a normal extension $\delta=\epsilon\circ\alpha$ for some $\alpha\in\Aut(L/K)$, and therefore since both $Z$ and $\Qcal$ are $\Aut(L/K)$-invariant, the second identity in $(\ref{LambdaW})$ follows.  The first identity in $(\ref{LambdaW})$ follows from the the local degree formula and the fact that the right-hand-side is independent of $w|v$.

Define a function $F_v:\Esf_v\setminus\epsilon(\Qcal)\to\RR$ by $F_v(x)=\frac{1}{|\Qcal|}\sum_{R\in \Qcal}\lambda_v(x-\epsilon(R))$.  Note that $F_v$ is continuous on $\Esf_v\setminus\epsilon(\Qcal)$ but has logarithmic singularities at the points of $\epsilon(\Qcal)$.  To get around this we define the truncated function $G_v:\Esf_v\to\RR$ by
\begin{equation}\label{TruncatedNeron}
G_v(x)=\frac{1}{|\Qcal|}\sum_{R\in \Qcal}\{j_v(x,R) + i_v^*(x,R)\} ,
\end{equation}
where $i_v^*(x,R)=\min\{i_v(x,R),m_v(Q)\}$, and where $m_v(Q)$ is defined in $(\ref{MconstantGlobal})$.  Therefore, using the identity $i_v(P,Q)=-\log d_v(P,Q)$ from $\S$~\ref{NeronFunction}, we still have
\begin{equation}\label{AgreeOnTorsion}
G_v(P) = F_v(P) \hskip1cm \text{ for all } P\in E(\KK_v)_\tor.
\end{equation} 
Note that $0\leq m_v(Q)<+\infty$ since $Q$ is nontorsion.  By its definition, $m_v(Q)$ is the minimal constant at which to truncate $i_v(x,R)$ such that $(\ref{AgreeOnTorsion})$ holds.  Moreover, $F_v$ and $G_v$ agree on the skeleton $\Sigma_v$ of $\Esf_v$.  In particular $\int G_v d\mu_v = \int F_v d\mu_v =0$, since the canonical measure $\mu_v$ is supported on $\Sigma_v$, and $\int\lambda_v(x,y)d\mu_v(x)=\int j_v(x,y)d\mu_v(x)=\int_{\Sigma_v} j_v(x,y)dm_{\Sigma_v}(x)=0$ for all $y\in \Esf_v$; here we have used that $i_v(x,y)=0$ for $x\in\Sigma_v$.  

For each $R\in\Qcal$, let $\kappa_R\in\Esf_v$ be the unique point in the (infinitely long) line segment $[r_{\Sigma_v}(\epsilon(R), \epsilon(R)]$ which satisfies $\rho(r_{\Sigma_v}(\epsilon(R)), \kappa_R)=m_v(Q)$.  Let $\Gamma_v=\Sigma_v\cup(\cup_{R\in\Qcal}[r_{\Sigma_v}(\epsilon(R)),\kappa_R])$ be the smallest connected subgraph of $\Esf_v$ containing $\Sigma_v$ and all of the points $\kappa_R$.  By Lemma \ref{LocalTruncLemma}, $G_v$ is a test function in the space $\Scal_{\Gamma_v}(\Esf_v,\RR)$ associated to the connected metrized subgraph $\Gamma_v$ of $\Esf_v$, 
\begin{equation}\label{TruncTestBound}
\|G'_v\|_{L^2(\Gamma_v)}^2 \leq \frac{1}{12}\log^+|j_E|_v + m_v(Q) \leq \frac{h(j_E)}{12} + m_v(Q),
\end{equation} 
and $\ell_0(\Gamma_v)=m_v(Q)$.  Applying $(\ref{LambdaW})$, Theorem \ref{QuanSUZ}, and the fact that $\int G_v d\mu_v=0$ and $\hhat(Z)=0$, we have
\begin{equation*}
\begin{split}
\Lambda_v(Z,\Qcal) & = \frac{1}{|Z|}\sum_{P\in Z}F_v(\epsilon(P)) \\
	& = \frac{1}{|Z|}\sum_{P\in Z}G_v(\epsilon(P)) \\
	& \leq \|G'_v\|_{L^2(\Gamma_v)} \bigg(\frac{h(j_E)}{12|Z|} + \frac{\ell_0(\Gamma_v)}{|Z|}\bigg)^{1/2} \\
	& \leq \frac{1}{|Z|^{1/2}}\bigg(\frac{h(j_E)}{12} + m_v(Q) \bigg),
\end{split}
\end{equation*}
which is $(\ref{Claim2})$.  The proof is complete.
\end{proof}

%%%%%%%%%%%%%%%%%%%%%%%%%%%%%%%%%%%%%%%%%%%%%%%%%%%%%%%%%%%%
%%%%%%%%%%%%%%%%%%%%%%%%%%%%%%%%%%%%%%%%%%%%%%%%%%%%%%%%%%%%
%%%%%%%%%%%%%%%%%%%%%%%%%%%%%%%%%%%%%%%%%%%%%%%%%%%%%%%%%%%%

\medskip

\end{document}